\documentclass[reqno]{amsart}

\usepackage{amsmath}
\usepackage{amsthm}
\usepackage{amssymb}
\usepackage{mathrsfs} 
\usepackage{bm}
\usepackage{courier}
\usepackage{array}
\usepackage{fnpct}
\usepackage{stmaryrd}
\usepackage[dvipsnames]{xcolor}
\usepackage{subcaption}
\usepackage{stmaryrd}
\usepackage{hyperref}
    \hypersetup{
        colorlinks,
        linkcolor={blue!50!white},
        citecolor={blue!50!white},
        urlcolor={blue!80!white}
    }
\usepackage{comment}
\usepackage{tikz,pgfplots}
\usetikzlibrary{arrows.meta}
\usetikzlibrary{shapes.misc}
\usepackage{ifthen}
\usepackage{tkz-graph}

\pgfdeclarelayer{background}
\pgfdeclarelayer{foreground}
\pgfsetlayers{background,main,foreground}

\numberwithin{equation}{section}

\newcommand{\R}{\mathbb{R}}

\newcommand{\Z}{\mathbb{Z}}

\renewcommand{\P}{\mathcal{P}}
\newcommand{\G}{\mathcal{G}}

\newcommand{\diam}{\mathrm{diam}}

\newcommand{\supp}{\mathrm{supp}}

\newcommand{\gp}{\mathrm{gp}}

\newcommand{\cprod}{\,\Box\,}

\newcommand{\floor}[1]{\left\lfloor #1\right\rfloor}
\newcommand{\ceil}[1]{\left\lceil #1\right\rceil}

\newcommand{\restrict}{\upharpoonright}

\newcommand{\spn}{\textrm{span}}
\renewcommand{\d}{d^*}

\renewcommand{\emptyset}{\varnothing}
\newcommand{\st}{:}

\newtheorem{theorem}{Theorem}[section]
\newtheorem{lemma}[theorem]{Lemma}
\newtheorem{corollary}[theorem]{Corollary}
\newtheorem{proposition}[theorem]{Proposition}

\theoremstyle{definition}

\newtheorem{question}[theorem]{Question}

\newtheorem{problems}[theorem]{Open Problems}
\newtheorem{conjecture}[theorem]{Conjecture}
\newtheorem{remark}[theorem]{Remark}

\newtheorem{definition}[theorem]{Definition}
\newtheorem{example}[theorem]{Example}

\date{\today}

\begin{document}

\title{The $k$-general $d$-position problem for graphs}

\author[Brent Cody]{Brent Cody}
\address[Brent Cody]{ 
Virginia Commonwealth University,
Department of Mathematics and Applied Mathematics,
1015 Floyd Avenue, PO Box 842014, Richmond, Virginia 23284, United States
} 
\email[B. ~Cody]{bmcody@vcu.edu} 
\urladdr{http://www.people.vcu.edu/~bmcody/}

\author[Garrett Moore]{Garrett Moore}
\address[Garrett Moore]{ 
Virginia Commonwealth University,
Department of Mathematics and Applied Mathematics,
1015 Floyd Avenue, PO Box 842014, Richmond, Virginia 23284, United States
} 
\email[G. ~Moore]{mooregk@vcu.edu} 

\thanks{}

\begin{abstract}
A set of vertices of a graph is said to be in general position if no three vertices from the set lie on a common geodesic. Recently Klav\v{z}ar, Rall and Yero generalized this notion by defining a set of vertices to be in general $d$-position if no three vertices from the set lie on a common geodesic of length at most $d$. We generalize this notion further by defining a set of vertices to be in $k$-general $d$-position if no $k$ vertices of the set lie on a common geodesic of length at most $d$. The $k$-general $d$-position number of a graph is the largest cardinality of a $k$-general $d$-position set. We provide upper and lower bounds on the $k$-general $d$-position number of graphs in terms of the $k$-general $d$-position number of certain kinds of subgraphs. We compute the $k$-general $d$-position number of finite paths and cycles. Along the way we establish that the maximally even subsets of cycles, which were introduced in Clough and Douthett's work on music theory, provide the largest possible $k$-general $d$-position sets in $n$-cycles. We generalize Klav\v{z}ar and Manuel's notion of monotone-geodesic labeling to that of $k$-monotone-geodesic labeling in order to calculate the $k$-general $d$-position number of the infinite two-dimensional grid. We also prove a formula for the $k$-general $d$-position number of certain thin finite grids, providing a partial answer to a question asked by Klav\v{z}ar, Rall and Yero.
\end{abstract}



\maketitle


\section{Introduction}\label{section_intro}


The classical no-three-in-line problem, posed by Dudeney \cite{MR0105345} over one-hundred years ago asks: What is the largest number of points that can be placed in the $n\times n$ grid so that no three points lie on the same line? Despite being widely studied for many years \cite{MR1168162, MR1492871, MR0238765, MR0366817, MR3404483}, this problem remains open. A graph-theoretic variation of the no-three-in-line problem called the \emph{general position problem for graphs}, introduced independently by Klav\v{z}ar and Manuel \cite{MR3849577} and by Chandran and Parthasarathy \cite{chandran2016geodesic}, asks: Given a graph $G$, what is the largest number of vertices that can be selected so that no three lie in a common shortest path?
This problem for graphs has itself been extensively studied by many authors \cite{MR3948765, MR4281067, MR4711379, MR4265041, MR4154901, MR4019752, MR3849577, MR3879620, MR4279308}. A generalization of the classical no-three-in-line problem, recently studied in \cite{Lefmann2024}, can be obtained by replacing ``three'' with an arbitrary integer parameter: What is the largest number of points that can be placed in the $n\times n$ grid so that no $k$ points lie on the same line? Such variations for arbitrary point sets in $\R^2$ have also been studied by Erd\H{o}s \cite{MR0852104, MR0974622} and many others \cite{MR2034714, MR1949900, MR3111653, MR3102591}. In this article, we introduce a variation of the general position problem for graphs in which sets of vertices may have three vertices in a common shortest path, but some larger number of vertices being in a common shortest path is forbidden. Before we formally state the problem, let us review some definitions.

Suppose $G=(V(G),E(G))$ is a graph. The \emph{length} of a path $g$ of $G$, written as $\lambda_G(g)$, or as $\lambda(g)$ when the context is clear, is the number of edges in $g$. For vertices $u,v\in V(G)$, the length of a shortest path from $u$ to $v$ in $G$ is denoted by $d_G(u,v)$, or just $d(u,v)$ when the context is clear. A path $g$ in $G$ is called a \emph{geodesic of $G$} or a \emph{shortest path of $G$} if it is a shortest path between two vertices of $G$. The set of vertices corresponding to some geodesic $g$ of $G$ will be denoted by $V(g)$. Notice that for any path $g$ of $G$, $|V(g)|=\lambda(g)+1$. We let $\G(G)$ denote the set of all geodesics of $G$.

Using the terminology of Klav\v{z}ar and Manuel \cite{MR3849577}, a set of vertices $S\subseteq V(G)$ is said to be a \emph{general position set in $G$} provided that 
\[\text{for all $g\in\G(G)$ we have $|S\cap V(g)|<3$};\] that is, no three vertices in $S$ lie in a common geodesic. The \emph{general position problem} for $G$ is to find a largest general position set of $G$. The \emph{general position number of $G$}, written as $\gp(G)$, is the largest cardinality of a general position set. 

Klav\v{z}ar, Rall and Yero \cite{MR4341189} generalized the notion of general position set as follows. For an integer $d\geq 1$, we say that a set of vertices $S\subseteq V(G)$ is a \emph{general $d$-position set in $G$} provided that
\[\text{for all $g\in\G(G)$, if $|S\cap V(g)|\geq 3$ then $\lambda(g)>d$}.\]
The \emph{general $d$-position problem for $G$} is to find a largest general $d$-position set in $G$, and the \emph{general $d$-position number of $G$}, written as $\gp_d(G)$, is the largest cardinality of a general $d$-position set. In the current article, we study a further generalization of general position sets.

\begin{definition} For integers $d$ and $k$ with $d\geq 1$ and $k\geq 2$, we say that a set of vertices $S\subseteq V(G)$ is a \emph{$k$-general $d$-position set in $G$}, or is \emph{in $k$-general $d$-position in $G$} (where \emph{in $G$} may be omitted when the context is clear) provided that
\begin{align}\text{for all $g\in\G(G)$, if $|S\cap V(g)|\geq k$ then $\lambda(g)>d$.}\label{definition_kgdp}\end{align}
The \emph{$k$-general $d$-position problem for $G$} is to find a largest $k$-general $d$-position set in $G$. The \emph{$k$-general $d$-position number of $G$}, denoted by $\gp^k_d(G)$, is the largest cardinality of a $k$-general $d$-position set in $G$.
\end{definition}

A few easy observations are in order. Recall that a set of vertices of a graph is \emph{independent} if no edge of the graph connects two vertices of the set. The \emph{independence number of $G$}, written as $\alpha(G)$, is the largest size of an independent set in $G$. It is easy to see that a set of vertices $A$ is $2$-general $1$-position in $G$ if and only if it is independent. Thus $\gp^2_1(G)=\alpha(G)$, and since a set being in $2$-general $d$-position, for some $d\geq 1$, implies that it is in $2$-general $1$-position, it follows that $2$-general $d$-position sets can be thought of as having a strong form of independence. Indeed, a set of vertices is in $2$-general $d$-position if and only if it is \emph{$d$-distance independent} in the terminology of \cite{MR3818423}. Hence, the \emph{$d$-distance independence number $\alpha_d(G)$} \cite{MR3818423}, which is the largest size of a $d$-distance independent set, satisfies
\[\gp^2_d(G)=\alpha_d(G).\]

For $d\geq 1$, since a set of vertices $S$ is a $3$-general $d$-position set in $G$ if and only if it is a general $d$-position set, we have
\[\gp^3_d(G)=\gp_d(G).\]
Furthermore, for $k\geq 2$ and $d\geq\diam(G)$, a set of vertices $S\subseteq V(G)$ is a $k$-general $d$-position set in $G$ if and only if 
\begin{align}\text{for all $g\in\G(G)$ we have $|S\cap V(g)|<k$}\label{equation_gkp}\end{align}
(note the independence of this notion from $d$). Thus,
\[\text{if $d\geq \diam(G)$ then $\gp^k_d(G)=\gp^k_{\diam(G)}(G)$}.\]
and we will often write $\gp^k(G)$ to mean $\gp^k_{\diam(G)}(G)$ in this case to emphasize the independence from $d$. Indeed, a set $S\subseteq V(G)$ is said to be a \emph{$k$-general position set} provided that it is a $k$-general $\diam(G)$-position set in $G$, i.e., (\ref{equation_gkp}) holds.

Notice that if $G$ has at least $k-1$ vertices then $\gp^k(G)\geq k-1$. For $d\geq 1$, $k\geq 2$ and $n\geq 1$ we have
\[\gp^k_d(K_n)=\begin{cases}
    1 & \text{if $k=2$}\\
    n & \text{if $k\geq 3$,}
\end{cases}\]
where $K_n$ is the complete graph on $n$ vertices.

In Section \ref{section_general}, we discuss some general properties of $k$-general $d$-position sets, some of which will be used throughout the paper. We establish an ordering that holds between the values of $\gp^k_d(G)$ in general (see Proposition \ref{proposition_lattice} and Figure \ref{figure_lattice}). We prove some results on $k$-general $d$-position sets in infinite graphs (see Corollary \ref{corollary_infinite}, Example \ref{example_infinite} and Proposition \ref{proposition_infinite}). In Lemma \ref{lemma_isometric_upper_bound}, we generalize Klav\v{z}ar and Manuel's \emph{isometric cover lemma} \cite[Theorem 3.1]{MR3849577} from general position sets to $k$-general $d$-position sets to obtain an upper bound on $\gp^k_d(G)$ in terms of the $k$-general $d$-position numbers of certain subgraphs. In Lemma \ref{lemma_isometric_lower_bound}, we generalize a result of Klav\v{z}ar, Rall and Yero \cite[Proposition 1.1]{MR4341189} to obtain an upper bound on $\gp^k_d(G)$ in terms of the $k$-general $d$-position number of certain subgraphs of $G$.

Klav\v{z}ar, Rall and Yero \cite{MR4341189} gave a formula for $\gp_d(P_n)$, where $P_n$ is a path on $n$ vertices, and $\gp_d(C_n)$, where $C_n$ is an $n$-cycle (see Proposition \ref{proposition_KRY_paths} and Remark \ref{remark_mistake}, respectively). Generalizing these results, in Section \ref{section_paths}, we prove Theorem \ref{theorem_paths} which provides a formula for $\gp^k_d(P_n)$ and in Section \ref{section_cycles}, we prove Theorem \ref{theorem_formula_for_C_n} which gives a formula for $\gp^k_d(C_n)$. Let us point out that the $k$-general $d$-position sets in $C_n$ taking on the maximal cardinality $\gp^k_d(C_n)$ are the maximally even sets, which were introduced by Clough and Douthett \cite{CloughDouthett} in their work on music theory and further studied in \cite{BCFW, BCL, MR2512671,  MR2366388, MR2408358, Taslakian_dissertation, MR2212108}. In Remark \ref{remark_mistake} we point out a minor error in the formula for $\gp_d(C_n)$ given in \cite{MR4341189}.

Let $P_\infty$ denote the two-way infinite path graph. Recall that the \emph{Cartesian product} of two graphs $G$ and $H$ is the 
graph $G\,\Box\, H$, with vertex set $V(G\,\Box\, H)=V(G)\times V(H)$, where two vertices $(u,v)$ and $(u',v')$ of $G\,\Box\, H$ are adjacent if $u=u'$ and $vv'\in E(H)$ or $uu'\in E(G)$ and $v=v'$. We let $P_\infty^{\Box, 2}=P_\infty\cprod P_\infty$ denote the usual cartesian product of two infinite paths. Klav\v{z}ar and Manuel \cite{MR3879620} introduced the notion of \emph{monotone-geodesic labeling} to prove that $\gp(P_\infty^{\Box,2})=4$. In Section \ref{section_infinite_2d_grids}, we generalize Klav\v{z}ar and Manuel's notion to that of \emph{$k$-monotone-geodesic labeling} (Definition \ref{definition_k_mon_geo_lab}) and prove Corollary \ref{corollary_infinite_grid}, which states that $\gp^k(P_\infty^{\Box,2})=(k-1)^2$.

In Section \ref{section_thin_grids}, we provide a partial answer to a question posed by Klav\v{z}ar, Rall and Yero, \cite[Section 6.1]{MR4341189}: they asked for the value of $\gp_d(P_n\cprod P_m)$ for $d\geq 3$. In Section \ref{section_thin_grids}, we provide formulas for $\gp^k_d(P_n\cprod P_2)$ for all meaningful values of $d, k$ and $n$.

In Section \ref{section_questions}, we discuss several open questions.

\section{General results on $k$-general $d$-position sets}\label{section_general}

Let us begin this section by pointing out that when $k$ is large enough so that $k$ vertices will not fit into a geodesic of length $d$, then every set of vertices is $k$-general $d$-position and hence $\gp^k_d(G)=|V(G)|$.

\begin{lemma}\label{lemma_all_vertices}
    Suppose $G$ is a finite connected graph. Let $d$ and $k$ be integers such that $2\leq k\leq \diam(G)+1$ and $1\leq d\leq k-2$. Then $V(G)$ is a $k$-general $d$-position set in $G$ and hence
    \[\gp^k_d(G)=|V(G)|.\]
\end{lemma}

\begin{proof}
If $g$ is a geodesic of $G$ containing at least $k$ vertices of $V(G)$, then $\lambda_G(g)\geq k-1>d$.
\end{proof}

An ordering among the quantities $\gp^k_d(G)$ for various values of $k\geq 2$ and $d\geq 1$ that holds in general for all graphs can be established as follows. Notice that if $S$ is a $k$-general $d$-position set, one can decrease $d$, increase $k$ or do both, and the definition will still be satisfied by $S$; put another way, if $S$ is a $k$-general $d$-position set, $k'\geq k$ and $d'\leq d$, then $S$ is a $k'$-general $d'$-position set. Hence, assuming $D\leq \diam(G)$ is finite, we obtain the lattice of inequalities displayed in Figure \ref{figure_lattice}, where the chain of inequalities in the second column of Figure \ref{figure_lattice} was previously noted \cite[Section 1]{MR4341189}. Thus, after introducing the parameter $k$, the values of the general $d$-position number $\gp_d(G)=\gp^3_d(G)$ studied in \cite{MR4341189} and the values of the $d$-distance independence number $\alpha_d(G)=\gp^2_d(G)$ studied in \cite{MR3818423}, can be seen as a one dimensional slices of a more general two-dimensional phenomenon.

\begin{proposition}\label{proposition_lattice}
Suppose $G$ is a graph. Let $1\leq d,d'\leq \diam(G)$ and $k,k'\geq 2$ be integers such that $d'\leq d$ and $k'\geq k$. Then $\gp^k_d(G)\leq\gp^{k'}_{d'}(G)$ (see Figure \ref{figure_lattice}).
\end{proposition}

Klav\v{z}ar, Rall and Yero \cite[Proposition 6.1]{MR4341189} proved that if $G$ is an infinite graph and $1\leq d<\infty$, then $\gp^3_d(G)=\infty$. Hence, we obtain the following corollary of \cite[Proposition 6.1]{MR4341189} and Proposition \ref{proposition_lattice}, which implies that the problem of determining $\gp^k_d(G)$ is trivial for infinite graphs when $k\geq 3$.

\begin{corollary}\label{corollary_infinite}
Suppose $G$ is an infinite graph and let $d\geq 2$ and $k\geq 3$ be integers. Then $\gp^k_d(G)=\infty$.
\end{corollary}

However, when $k=2$, the quantity $\gp^2_d(G)$ is not infinite for all infinite graphs. For example, when $K_\kappa$ is an infinite complete graph with cardinality $|V(G)|=\kappa$ we have $\gp^2_d(K_\kappa)=1$. Evidently, as demonstrated in the next example, the $2$-general $d$-position number of infinite graphs with cardinality $\kappa$ can take on many different values. 

\begin{example}\label{example_infinite}
Let $G$ be the graph obtained by connecting one of the endpoints of $P_n$ with one of the vertices $v$ of $K_\kappa$. Since $\gp^2_d(P_n)=\ceil{\frac{n}{d+1}}$ (this follows directly from Theorem \ref{theorem_paths}), and since we can essentially treat $v$ and $V(K_n)\setminus \{v\}$ as being two additional vertices, we have \[\gp^2_d(G)=\gp^2_d(P_{n+2})=\ceil{\frac{n+2}{d+1}}.\]

\end{example}

However, for some infinite graphs, the problem of finding the $2$-general $d$-position number does trivialize, as indicated in the next---self-evident---proposition.

\begin{proposition}\label{proposition_infinite}
Suppose that either the diameter of $G$ is infinite or $G$ has infinitely many connected components. Then $\gp^2_d(G)=\infty$ for all $d\geq 1$.
\end{proposition}

\begin{figure}
\begin{center}
\resizebox{\textwidth}{!}{
\begin{tikzpicture}
\def \a {1.2};

\node[] at (0*\a,0) {$\gp(G)$};
\node[] at (1*\a,0) {$\leq$};
\node[] at (2*\a,0) {$\gp^4_D(G)$};
\node[] at (3*\a,0) {$\leq$};
\node[] at (4*\a,0) {$\gp^5_D(G)$};
\node[] at (5*\a,0) {$\leq$};
\node[] at (6*\a,0) {$\cdots$};
\node[] at (7*\a,0) {$\leq$};
\node[] at (8*\a,0) {$\gp^D_D(G)$};
\node[] at (9*\a,0) {$\leq$};
\node[] at (10*\a,0) {$\gp^{D+1}_D(G)$};
\node[] at (11*\a,0) {$\leq$};
\node[] at (11.5*\a,0) {$|G|$};

\node[rotate=90] at (0*\a,0.5) {$\leq$};
\node[rotate=90] at (2*\a,0.5) {$\leq$};
\node[rotate=90] at (4*\a,0.5) {$\leq$};
\node[rotate=90] at (8*\a,0.5) {$\leq$};
\node[rotate=90] at (10*\a,0.5) {$\leq$};

\node[] at (0*\a,1) {$\gp^3_{D-1}(G)$};
\node[] at (1*\a,1) {$\leq$};
\node[] at (2*\a,1) {$\gp^4_{D-1}(G)$};
\node[] at (3*\a,1) {$\leq$};
\node[] at (4*\a,1) {$\gp^5_{D-1}(G)$};
\node[] at (5*\a,1) {$\leq$};
\node[] at (6*\a,1) {$\cdots$};
\node[] at (7*\a,1) {$\leq$};
\node[] at (8*\a,1) {$\gp^D_{D-1}(G)$};
\node[]  at (9*\a,1) {$\leq$};
\node[] at (10*\a,1) {$|G|$};

\node[rotate=90] at (0*\a,1.5) {$\leq$};
\node[rotate=90] at (2*\a,1.5) {$\leq$};
\node[rotate=90] at (4*\a,1.5) {$\leq$};
\node[rotate=90] at (8*\a,1.5) {$\leq$};

\node[] at (0*\a,2) {$\gp^3_{D-2}(G)$};
\node[] at (1*\a,2) {$\leq$};
\node[] at (2*\a,2) {$\gp^4_{D-2}(G)$};
\node[] at (3*\a,2) {$\leq$};
\node[] at (4*\a,2) {$\gp^5_{D-2}(G)$};
\node[] at (5*\a,2) {$\leq$};
\node[] at (6*\a,2) {$\cdots$};
\node[] (L1) at (7*\a,2) {$\leq$};
\node[] at (8*\a,2) {$|G|$};

\node[rotate=90] at (0*\a,2.5) {$\leq$};
\node[rotate=90] at (2*\a,2.5) {$\leq$};
\node[rotate=90] at (4*\a,2.5) {$\leq$};

\node[] at (0*\a,3.2) {$\vdots$};
\node[] at (2*\a,3.2) {$\vdots$};
\node[] at (4*\a,3.2) {$\vdots$};
\node[rotate=45] at (6*\a,3) {$\vdots$};

\node[rotate=90] at (0*\a,3.6) {$\leq$};
\node[rotate=90] at (2*\a,3.6) {$\leq$};
\node[rotate=90] at (4*\a,3.6) {$\leq$};
\node[rotate=90] at (6*\a,3.6) {$\leq$};

\node[] at (0*\a,4.2) {$\gp^3_4(G)$};
\node[] at (1*\a,4.2) {$\leq$};
\node[] at (2*\a,4.2) {$\gp^4_4(G)$};
\node[] at (3*\a,4.2) {$\leq$};
\node[] at (4*\a,4.2) {$\gp^5_4(G)$};
\node[] (L2) at (5*\a,4.2) {$\leq$};
\node[] at (6*\a,4.2) {$|G|$};

\node[rotate=90] at (0*\a,4.7) {$\leq$};
\node[rotate=90] at (2*\a,4.7) {$\leq$};
\node[rotate=90] at (4*\a,4.7) {$\leq$};


\node[] at (0*\a,5.2) {$\gp^3_3(G)$};
\node[] at (1*\a,5.2) {$\leq$};
\node[] at (2*\a,5.2) {$\gp^4_3(G)$};
\node[] at (3*\a,5.2) {$\leq$};
\node[] at (4*\a,5.2) {$|G|$};

\node[rotate=90] at (0*\a,5.7) {$\leq$};
\node[rotate=90] at (2*\a,5.7) {$\leq$};

\node[] at (0*\a,6.2) {$\gp^3_2(G)$};
\node[] at (1*\a,6.2) {$\leq$};
\node[] at (2*\a,6.2) {$|G|$};

\node[rotate=90] at (0*\a,6.7) {$\leq$};

\node[] at (0*\a,7.2) {$|G|$};


\draw[rounded corners] (-0.9, -0.5) rectangle (0.9, 7.7) {};

\node[] at (-2*\a,0) {$\gp^2_D(G)$};
\node[rotate=90] at (-2*\a,0.5) {$\leq$};
\node[] at (-2*\a,1) {$\gp^2_{D-1}(G)$};
\node[rotate=90] at (-2*\a,1.5) {$\leq$};
\node[] at (-2*\a,2) {$\gp^2_{D-2}(G)$};
\node[rotate=90] at (-2*\a,2.5) {$\leq$};
\node[] at (-2*\a,3.2) {$\vdots$};
\node[rotate=90] at (-2*\a,3.6) {$\leq$};
\node[] at (-2*\a,4.2) {$\gp^2_4(G)$};
\node[rotate=90] at (-2*\a,4.7) {$\leq$};
\node[] at (-2*\a,5.2) {$\gp^2_3(G)$};
\node[rotate=90] at (-2*\a,5.7) {$\leq$};
\node[] at (-2*\a,6.2) {$\gp^2_2(G)$};
\node[rotate=90] at (-2*\a,6.7) {$\leq$};
\node[] at (-2*\a,7.2) {$\gp^2_1(G)$};
\node[rotate=90] at (-2*\a,7.7) {$\leq$};
\node[] at (-2*\a,8.2) {$|G|$};

\node[] at (-1*\a,0) {$\leq$};
\node[] at (-1*\a,1) {$\leq$};
\node[] at (-1*\a,2) {$\leq$};
\node[] at (-1*\a,4.2) {$\leq$};
\node[] at (-1*\a,5.2) {$\leq$};
\node[] at (-1*\a,6.2) {$\leq$};
\node[] at (-1*\a,7.2) {$\leq$};

\end{tikzpicture}
}
\end{center}
\caption{A lattice of inequalities holds in general.}
\label{figure_lattice}
\end{figure}
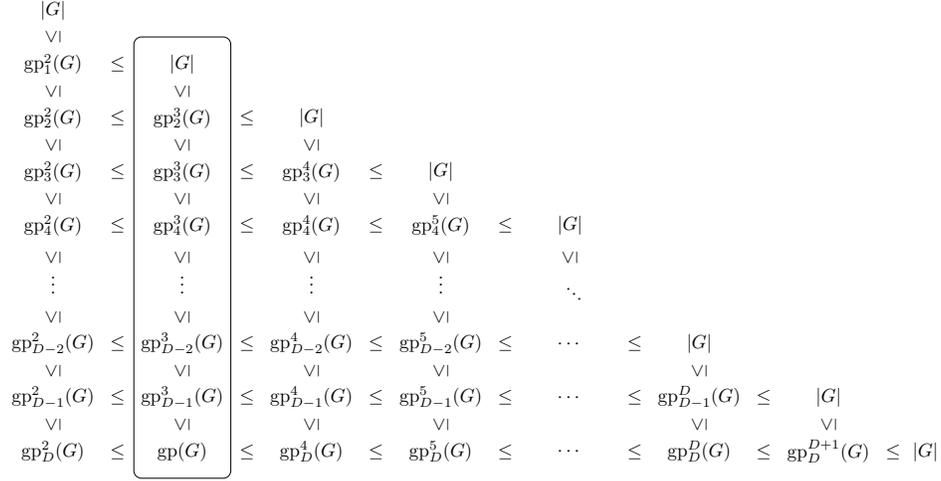

Let us now generalize Klav\v{z}ar and Manuel's \emph{isometric cover lemma} \cite[Theorem 3.1]{MR3849577} from the general position number to the $k$-general $d$-position number. A subgraph $H$ of a graph $G$ is \emph{convex subgraph of $G$} if it includes every shortest path in $G$ between two of its vertices. A subgraph $H$ of $G$ is said to be an \emph{isometric subgraph of $G$} if $d_H(u,v)=d_G(u,v)$ for all $u,v\in V(H)$. Clearly, every convex subgraph of a graph $G$ is an isometric subgraph.

\begin{lemma}\label{lemma_isometric_down_convex_down_and_up}
    Suppose $G$ is a connected graph and $H$ is a subgraph of $G$. Let $d\geq 1$ and $k\geq 2$ be integers.
    \begin{enumerate}
        \item If $H$ is an isometric subgraph of $G$ then whenever $A\subseteq V(H)$ is a $k$-general $d$-position set in $G$, it follows that $A$ is a $k$-general $d$-position set in $H$.
        \item If $H$ is a convex subgraph of $G$ then for all $A\subseteq V(H)$, $A$ is a $k$-general $d$-position set in $G$ if and only if it is a $k$-general $d$-position set in $H$.
    \end{enumerate}
\end{lemma}

\begin{proof}
Suppose $H$ is an isometric subgraph of $G$ and $A\subseteq V(H)$ is a $k$-general $d$-position set in $G$. Suppose $g\in\G(H)$ is a geodesic of $H$ with $|A\cap g|\geq k$. Since $H$ is an isometric subgraph of $G$, we see that $g\in\G(G)$ and furthermore, since $A$ is a $k$-general $d$-position set in $G$ we have $\lambda_G(g)>d$. Again, applying the fact that $H$ is an isometric subgraph of $G$, we have $\lambda_H(g)=\lambda_G(g)>\lambda$, as desired.

Suppose $H$ is a convex subgraph of $G$ and $A\subseteq V(H)$. Since $H$ is an isometric subgraph of $G$, it follows from (1) that if $A$ is a $k$-general $d$-position set in $G$ then it is a $k$-general $d$-position set in $H$. For the converse, suppose $A$ is a $k$-general $d$-positions set in $H$. Let $g\in\G(G)$ is a geodesic of $G$ with $|A\cap g|\geq k$. Let $g=\{v_1,v_2\ldots,v_\ell\}$ where $v_1v_2\cdots v_\ell$ is a shortest path in $G$. Let $a$ be the least index of an element of $A$ in $g$ and let $b$ be the greatest index of an element of $A$ in $g$. Then $v_av_{a+1}\cdots v_b$ is a shortest path in $G$ between two elements of $H$. Since $H$ is a convex subgraph of $G$, $v_av_{a+1}\cdots v_b$ is a shortest path in $H$ and so $g'=\{v_a,\ldots,v_b\}$ is a geodesic of $H$ such that $|A\cap g'|\geq k$. Since $A$ is a $k$-general $d$-position set in $H$ we have $\lambda_H(g')>d$, and since $H$ is an isometric subgraph of $G$ we have $\lambda_G(g)\geq \lambda_G(g')=\lambda_H(g')>d$. So, $A$ is a $k$-general $d$-position set in $G$.
\end{proof}

Generalizing \cite[Theorem 3.1]{MR3849577}, we obtain the following.

\begin{lemma}\label{lemma_isometric_upper_bound}
Suppose $G$ is a finite connected graph and $\{H_1,\ldots,H_\ell\}$ is a collection of isometric subgraphs of $G$ such that $V(G)=\bigcup_{i=1}^\ell V(H_i)$. Suppose $d$ and $k$ are integers with $n\geq 1$ and $k\geq 2$. Then
\[\gp^k_d(G)\leq\sum_{i=1}^\ell \gp^k_d(H_i).\]
\end{lemma}

\begin{proof}
Suppose $A$ is a $k$-general $d$-position set in $G$. By Lemma \ref{lemma_isometric_down_convex_down_and_up} (1), for all $i\in\{1,\ldots,\ell\}$, $A\cap V(H_i)$ is a $k$-general $d$-position set in $H_i$ and thus $|A\cap V(H_i)|\leq \gp^k_d(H_i)$. This implies that
\[|A|\leq\sum_{i=1}^\ell\gp^k_d(H_i),\]
as desired.
\end{proof}

Next, we generalize a result on general $d$-position sets due to Klav\v{z}ar, Rall and Yero \cite[Proposition 1.1]{MR4341189} to obtain a lower bound on $\gp^k_d(G)$ in terms of isometric subgraphs. Given a graph $G$ with subgraphs $H$ and $K$, we let 
\[d_G(H,K)=\min\{d_G(u,v)\st u\in V(H)\text{ and }v\in V(K)\}.\]

\begin{lemma}\label{lemma_isometric_lower_bound}
Suppose $d\geq 1$ and $k\geq 2$ are integers. Let $G$ be a finite connected graph and let $\{H_1,\ldots,H_\ell\}$ be a finite collection of isometric subgraphs of $G$ such that $d_G(H_i,H_j)\geq d$ for $i\neq j$. Then
\[\sum_{i=1}^\ell\gp^k_d(H_i)\leq\gp^k_d(G).\]
\end{lemma}

\begin{proof}
For each $i\in\{1,\ldots,\ell\}$, let $S_i\subseteq V(H_i)$ be a $k$-general $d$-position set in $H_i$ with $|S_i|=\gp^k_d(H_i)$. Let $S=\bigcup_{i=1}^\ell S_i$. Since the vertex sets of the subgraphs in $\{H_1,\ldots,H_\ell\}$ are pairwise disjoint, we have 
\[|S|=\sum_{i=1}^\ell\gp^k_d(H_i).\]
Let us show that $S$ is $k$-general $d$-position in $G$. Suppose not, then there is some geodesic $g$ of $G$ such that $|S\cap V(g)|\geq k$ and $\lambda_G(g)\leq d$. Since $d_G(H_i,H_j)\geq d$ for $i\neq j$, it follows that there is an $i_0\in\{1,\ldots,\ell\}$ such that $S\cap V(g)\subseteq V(H_{i_0})$, and hence $S\cap V(g)=S_{i_0}$. But, since $H_{i_0}$ is an isometric subgraph of $G$, it follows that there is a geodesic $g'$ of $H_{i_0}$ such that $V(g')=V(g)\cap V(H_{i_0})$ and $\lambda_{H_{i_0}}(g')\leq\lambda_G(g)$. Thus $|S_{i_0}\cap V(g')|\geq k$ and $\lambda_{H_{i_0}}(g')\leq d$, which contradicts the fact that $S_{i_0}$ is a $k$-general $d$-position set in $H_i$.
\end{proof}

\section{Paths}\label{section_paths}

In the current section we will generalize the following result of Klav\v{z}ar, Rall and Yero to $k$-general $d$-position sets.

\begin{proposition}[{Klav\v{z}ar, Rall and Yero, \cite[Proposition 3.1]{MR4341189}}]\label{proposition_KRY_paths}
If $n\geq 3$ and $2\leq d\leq n-1$, then
\[\gp_d(P_n)=\begin{cases}
    2\ceil{\frac{n}{d+1}}-1 &\text{if $n \equiv 1 \bmod (d+1)$}\\[0.5em]
    2\ceil{\frac{n}{d+1}} &\text{otherwise.}
\end{cases}\]
\end{proposition}

Let us first establish a formula for $\gp^k_d(P_n)$ under the assumption that $n\leq d+1$. Let us note that we will take the vertices of $P_n$ to be $V(P_n)=\{1,\ldots,n\}$ with adjacencies given by $\{i,i+1\}\in E(P_n)$ for $i\in\{1,\ldots,n-1\}$.

\begin{lemma}\label{lemma_paths_short}
    Suppose $d,k$ and $n$ are integers such that $d\geq 1$ and $2\leq k\leq n \leq d+1$. Then
    \[\gp^k_d(P_n)=\gp^k(P_n)=\min(n,k-1).\]
\end{lemma}

\begin{proof}
Since $n\leq d+1$, it follows that there are no geodesics in $P_n$ with length greater than $d$. Thus, $\gp^k_d(P_n)=\gp^k(P_n)$. It easily follows that 
\[A=\{1,\ldots,\min(n,k-1)\}\] is a $k$-general position set in $P_n$, and any set of vertices of $P_n$ with cardinality greater than $\min(n,k-1)$ is not a $k$-general position set in $P_n$. Thus, $\gp^k_d(P_n)=\gp^k(P_n)=\min(n,k-1)$.
\end{proof}

Now we will prove the general version of the formula for $\gp^k_d(P_n)$.

\begin{theorem}\label{theorem_paths}
    Suppose $d,k$ and $n$ are integers such that $d\geq 1$ and $2\leq k \leq n$. 
    Then
        \[\gp^k_d(P_n)=\begin{cases}
        n & \text{if $1\leq d\leq k-2$}\\
        (k-1)\floor{\frac{n}{d+1}}+\min(n\bmod(d+1),k-1) & \text{if $k-1\leq d$.}
    \end{cases}.\]
\end{theorem}

\begin{proof}
Fix $d,k$ and $n$ as in the hypothesis. The first case of the formula follows easily from Lemma \ref{lemma_all_vertices}.

Suppose $k-1\leq d$. We define a particular subset $S_{k,d}$ of $\Z_{\geq 1}$ that will be used to define a $k$-general $d$-position subset of $P_n$. Let
\[S_{k,d}=\bigcup_{i=0}^\infty[(d+1)i+1,(d+1)i+k-1].\]

Define $q=\floor{\frac{n}{d+1}}$ and let $r=n\bmod(d+1)$; in other words, let $r$ be the unique integer such that $n=q(d+1)+r$ and $0\leq r<d+1$. Let $m=\min(n\bmod(d+1),k-1)$. It is easy to verify that
\[|S_{k,d}\cap [1,n]|=(k-1)q+m.\]
The induced subgraph $P_n[S_{k,d}\cap [1,n]]$ has either $q$ or $q+1$ connected components, each of size at most $k-1$. Furthermore, the distance between the least elements of two distinct connected components of $P_n[S_{k,d}\cap [1,n]]$ is at least $d+1$. Thus, if a geodesic $g\in\G(P_n)$ contains $k$ or more vertices of $S_{k,d}\cap [1,n]$, then $\lambda_{P_n}(g)>d$. This implies that $S_{k,d}\cap [1,n]$ is a $k$-general $d$-position set in $P_n$ and hence $\gp^k_d(P_n)\geq (k-1)q+m$.

Let us show that $\gp^k_d(P_n)\leq (k-1)q+m$. 

Suppose $n<d+1$. Then $q=0$ and $n\bmod(d+1)=n$. Thus, it follows by Lemma \ref{lemma_paths_short}, that 
\[\gp^k_d(P_n)=\min(n,k-1)=(k-1)q+m.\] 

Suppose $n=d+1$, then $q=1$, $r=n\bmod (d+1)=m=0$, and it follows by Lemma \ref{lemma_paths_short}, that 
\[\gp^k_d(P_n)=\gp^k_d(P_{d+1})=\min(d+1,k-1)=k-1=(k-1)q+m.\]

Now suppose that $n>d+1$. Then $q\geq 1$. We define a sequence $H_0,\ldots,H_q$ of subgraphs of $P_n$ as follows. For $i\in\{0,\ldots,q-1\}$, let \[X_i=[(d+1)i+1,(d+1)(i+1)].\] If $r>0$ let $X_q=[(d+1)q+1,n]$ and otherwise, let $X_q=\emptyset$. For $i\in\{0,\ldots,q\}$ let $H_i=P_n[X_i]$ and note that $H_i$ is an isometric subgraph of $P_n$. Furthermore, $V(P_n)=\bigcup_{i=0}^qV(H_i)$. Hence by applying Lemma \ref{lemma_isometric_upper_bound}, and using that $H_i\cong P_{d+1}$ for $i\in\{0,\ldots,q-1\}$ and $H_q\cong P_r$ where $r<d+1$, it follows from Lemma \ref{lemma_paths_short} that 
\[\gp^k_d(P_n)\leq\sum_{i=0}^q\gp^k_d(H_i)=\gp^k_d(P_{d+1})q+\gp^k_d(P_r)=(k-1)q+m.\]
\end{proof}

\begin{example}
The values of $\gp^k_d(P_{14})$ are indicated in the table below. Note that the values in the $k=3$ column match those of $\gp_d(P_{14})$, which appear in \cite[Table 1]{MR4341189}. Thus, we see that the values of $\gp_d$, for $P_{14}$ and more broadly for any graph, are a one dimensional slice of a multidimensional phenomenon.

\begin{center}
\begin{tabular}{>{\centering\arraybackslash}p{5em}|cccccccccccccc}
 & \multicolumn{14}{c@{}}{values of $k$}\\
values of $d$ & 2 & 3 & 4 & 5 & 6 & 7 & 8 & 9 & 10 & 11 & 12 & 13 & 14 & 15\\
\hline
1 & 7 & 14 & 14 & 14 & 14 & 14 & 14 & 14 & 14 & 14 & 14 & 14 & 14 & 14\\
2 & 5 & 10 & 14 & 14 & 14 & 14 & 14 & 14 & 14 & 14 & 14 & 14 & 14 & 14\\
3 & 4 & 8 & 11 & 14 & 14 & 14 & 14 & 14 & 14 & 14 & 14 & 14 & 14 & 14\\
4 & 3 & 6 & 9 & 12 & 14 & 14 & 14 & 14 & 14 & 14 & 14 & 14 & 14 & 14\\
5 & 3 & 6 & 8 & 10 & 12 & 14 & 14 & 14 & 14 & 14 & 14 & 14 & 14 & 14\\
6 & 2 & 4 & 6 & 8 & 10 & 12 & 14 & 14 & 14 & 14 & 14 & 14 & 14 & 14\\
7 & 2 & 4 & 6 & 8 & 10 & 12 & 13 & 14 & 14 & 14 & 14 & 14 & 14 & 14\\
8 & 2 & 4 & 6 & 8 & 10 & 11 & 12 & 13 & 14 & 14 & 14 & 14 & 14 & 14\\
9 & 2 & 4 & 6 & 8 & 9 & 10 & 11 & 12 & 13 & 14 & 14 & 14 & 14 & 14\\
10 & 2 & 4 & 6 & 7 & 8 & 9 & 10 & 11 & 12 & 13 & 14 & 14 & 14 & 14\\
11 & 2 & 4 & 5 & 6 & 7 & 8 & 9 & 10 & 11 & 12 & 13 & 14 & 14 & 14\\
12 & 2 & 3 & 4 & 5 & 6 & 7 & 8 & 9 & 10 & 11 & 12 & 13 & 14 & 14\\
13 & 1 & 2 & 3 & 4 & 5 & 6 & 7 & 8 & 9 & 10 & 11 & 12 & 13 & 14\\
\end{tabular}
\captionof{table}{The quantities $\gp^k_d(P_{14})$ for various values of $k$ and $d$.}\label{sophisticatedtable}
\end{center}

\end{example}

\section{Cycles}\label{section_cycles}

\subsection{Preliminaries on maximal evenness}

Before stating our formula for $\gp^k_d(C_n)$, we discuss some preliminaries involving maximally even sets, introduced by Clough and Douthett \cite{CloughDouthett}. See \cite{BCL} for more background.

Let us note that we take the vertices of $C_n$ as being $V(C_n)=\{0,\ldots,n-1\}$ and edges $\{i,(i+1)\bmod n\}$ for each $i\in\{0,\ldots,n-1\}$. Fix vertices $u$ and $v$ in $C_n$. The \emph{clockwise distance} from $u$ to $v$ is the least nonnegative integer congruent to $(v-u)$ mod $n$, that is
\[d(u,v)=(v-u)\bmod n.\]
Recall that the \emph{distance} (or \emph{geodesic distance}) between $u$ and $v$ is the length of a shortest path from $u$ to $v$. For any set $A=\{a_0,\ldots,a_{m-1}\}$ of vertices in $C_n$ with $|A|=m\leq n$, arranged in increasing order, we define the \emph{$A$-span of $(a_i,a_k)$} to be
\[\spn_A(a_i,a_j)=(j-i)\bmod m.\]
We define two multisets: the \emph{$k$-multispectrum of clockwise distances} of $A$ is the multiset
\[\sigma^*_k(A)=\left[\,d^*(u,v)\st \spn_A(u,v)=k\,\right]\]
and the \emph{$k$-multispectrum of geodesic distances} of $A$ is the multiset
\[\sigma_k(A)=\left[\,d(u,v)\st \spn_A(u,v)=k\,\right].\]
For a multiset $X$ we define $\supp(X)$ to be the set whose elements are those of $X$.

\begin{definition}[{Clough and Douthett \cite{CloughDouthett}}]\label{definition_me} A set $A$ of vertices in $C_n$ with $|A|=m$ is \emph{maximally even} if for each integer $k$ with $1\leq k \leq m-1$ the set $\supp(\sigma^*(A))$ consists of either a single integer or two consecutive integers.
\end{definition}

\begin{definition}[{Clough and Douthett \cite{CloughDouthett}}]\label{definition_J}
Suppose $m$, $n$ and $r$ are integers with $1\leq m\leq n$ and $0\leq r\leq n-1$. We define
\[J^r_{n,m}=\left\{\floor{\frac{ni+r}{m}}\st i\in \Z \text{ and } 0\leq i < m\right\}.\]
Sets of the form $J^r_{n,m}$ are called \emph{$J$-representations}.
\end{definition}

\begin{lemma}[{Clough and Douthett \cite[Corollary 1.2]{CloughDouthett}}]\label{lemma_spec_of_J_rep_clockwise}
Suppose $k$, $m$, $n$ and $r$ are integers with $1\leq m\leq n$ and $0\leq r\leq n-1$. If $1\leq k\leq m-1$ then
\[\supp(\sigma^*_k(J^r_{n,m}))=\left\{\floor{\frac{nk}{m}},\ceil{\frac{nk}{m}}\right\}.\]
Hence $J^r_{n,m}$ is a maximally even subset of $C_n$.
\end{lemma}

Indeed, Clough and Douthett \cite{CloughDouthett} showed that a set of vertices $A$ of $C_n$ is maximally even if and only if it is of the form $J^r_{n,m}$ for some appropriate values of $m$, $n$ and $r$.

Prior to Clough and Douthett's work on maximal evenness, the following easy lemma was established by Clough and Myerson \cite[Lemma 2]{CloughAndMyerson}.

\begin{lemma}\label{lemma_sum_is_kn}
Suppose $n$ is a positive integer, $A$ is a set of vertices in $C_n$ with $|A|=m$ and $1\leq k\leq m-1$. Then
\[\sum \sigma^*_k(A)=kn.\]
\end{lemma}

\subsection{A new result on cycles}

Before stating our formula for the $k$-general $d$-position number of cycles, let us review a proposition from \cite{MR4341189}.

\begin{remark}\label{remark_mistake}
It is stated in \cite[Proposition 3.3]{MR4341189} that if $n\geq 5$ and $2\leq d<\floor{\frac{n}{2}}$, then
\begin{align}\gp_d(C_n)=\begin{cases}
    \displaystyle 2\floor{\frac{n}{d+1}}+1 &\text{if $n\equiv d \bmod (d+1)$}\\[1em]
    \displaystyle 2\floor{\frac{n}{d+1}} &\text{otherwise.}
\end{cases}\label{equation_mistake}
\end{align}
Furtermore, if  $d\geq\floor{\frac{n}{2}}$ then $\gp_d(C_n)=3$.

Let us point out a minor problem with (\ref{equation_mistake}). Take $n=16$ and $d=5$. Then, according to Proposition (\ref{equation_mistake}), we have $\gp_5(C_{16})=2\floor{\frac{16}{6}}=4$. However, the set 
\[A=J^0_{16,5}=\left\{\floor{\frac{16i}{5}}\st 0\leq i<5\right\}=\{0,3,6,9,12\}\]
is a general $5$-position set in $C_{16}$ with cardinality $5$, thus contradicting (\ref{equation_mistake}). Let us note that (\ref{equation_mistake}) can easily be corrected by moving the $2$ inside the floor function in the second case (this is a consequence of Proposition \ref{theorem_formula_for_C_n}).
\end{remark}

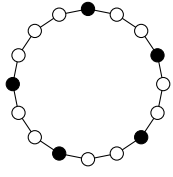
\begin{figure}
\centering
\begin{tikzpicture}[scale=0.5]

  \foreach \i in {0,3,6,9,12} {
    \coordinate (v\i) at ({90-(360/16)*\i}:2);
    \node[draw, circle, fill=black, text=white, inner sep=0pt, minimum width = 5pt] at (v\i) {};
  }
  
  \foreach \i in {1,2,4,5,7,8,10,11,13,14,15} {
    \coordinate (v\i) at ({90-(360/16)*\i}:2);
    \node[draw, circle, fill=white, inner sep=0pt, minimum width = 5pt] at (v\i) {};
  }
  
  \begin{pgfonlayer}{background}
    \foreach \i in {0,...,14} {
      \draw (v\i) -- (v\the\numexpr\i+1\relax);
    }
    \draw (v15) -- (v0); 
  \end{pgfonlayer}
\end{tikzpicture}

\caption{The maximally even set $J^0_{16,5}$.}
\label{figure_me_set}
\end{figure}

Generalizing \cite[Proposition 3.3]{MR4341189}, we obtain the following.

\begin{theorem}\label{theorem_formula_for_C_n}
    Suppose $d,k$ and $n$ are positive integers such that $1\leq d\leq \floor{\frac{n}{2}}$ and $2\leq k\leq \floor{\frac{n}{2}}+1$. Then
    \[\gp^k_d(C_n)=\begin{cases}
        \displaystyle n &\text{if $1\leq d\leq k-2$}.\\[0.5em]
        \displaystyle \floor{\frac{(k-1)n}{d+1}} &\text{if $k-1\leq d\leq\floor{\frac{n}{2}}$}.\\
    \end{cases}\]
\end{theorem}

\begin{proof}
Suppose $1\leq d\leq k-2$. It easily follows from Lemma \ref{lemma_all_vertices} that $V(C_n)$ is a $k$-general $d$-position set and hence $\gp^k_d(C_n)=n$. 

Suppose $k-2<d\leq\floor{\frac{n}{2}}$. Then, by our assumptions we have $3\leq k\leq\floor{\frac{n}{2}}+1$ and hence $2\leq k-1\leq \floor{\frac{n}{2}}$. Let $m=\floor{\frac{(k-1)n}{d+1}}$ and let $r$ be the unique integer such that $(k-1)n=m(d+1)+r$ and $0\leq r<d+1$.

First we will show that $J^0_{n,m}$ (see Definition \ref{definition_J}) is a $k$-general $d$-position set in $C_n$. Suppose $g$ is a geodesic of $C_n$ that contains at least $k$ vertices of $S$. Then, there are $s,s'\in S\cap V(g)$ such that $\spn_S(s,s')=k-1$ and, using Lemma \ref{lemma_spec_of_J_rep_clockwise},
\[d(s,s')=d^*(s,s')\in\supp(\sigma_{k-1}(J^0_{n,m}))=\left\{\floor{\frac{(k-1)n}{m}},\ceil{\frac{(k-1)n}{m}}\right\}.\]
We have
\[d(s,s')\geq \floor{\frac{(k-1)n}{m}}\geq d+1+\floor{\frac{r}{m}}\geq d+1,\]
and hence $\lambda_{C_n}(g)>d$. Thus, $J^0_{n,m}$ is a $k$-general $d$-position set in $C_n$ and hence $\gp^k_d(C_n)\geq m$. 

To see that $\gp^k_d(C_n)\leq m$, suppose that $S\subseteq V(C_n)$ and $|S|>m$. We will prove that $S$ is not a $k$-general $d$-position set in $C_n$. Let us show that there is some $a\in \sigma^*_{k-1}(S)$ with $a\leq d$. Suppose not, then all elements of $\sigma^*_{k-1}(S)$ are greater than $d$ and since $|S|>m$, by Lemma \ref{lemma_sum_is_kn}, we have
\[(k-1)n=\sum\sigma^*_{k-1}(S)\geq |S|(d+1)\geq (m+1)(d+1).\]
But then $(k-1)n=m(d+1)+r\geq m(d+1)+d+1,$
which implies $r\geq d+1$, a contradiction. Let $a\in\sigma^*_{k-1}(S)$ be such that $a\leq d$. Then there is a pair $(s,s')\in S\times S$ with $d(s,s')=d^*(s,s')=a\leq d$ such that a shortest path from $s$ to $s'$ contains $k$ elements of $S$. Hence $S$ is not a $k$-general $d$-position set in $C_n$.
\end{proof}

\begin{corollary}\label{corollary_me_witness}
Suppose $d,k,m$ and $n$ are positive integers such that $1\leq d\leq \floor{\frac{n}{2}}$ and $3\leq k\leq \floor{\frac{n}{2}}+2$. Given that $\gp^k_d(C_n)=m$, it follows that $J^0_{n,m}$ is a $k$-general $d$-position set in $C_n$.
\end{corollary}

\begin{center}
\begin{tabular}{>{\centering\arraybackslash}p{5em}|cccccccc}
 & \multicolumn{8}{c@{}}{values of $k$}\\
values of $d$  & 2 & 3 & 4 & 5 & 6 & 7 & 8 & 9\\
\hline
1 & 7 & 14 & 14 & 14 & 14 & 14 & 14 & 14\\
2 & 4 & 9 & 14 & 14 & 14 & 14 & 14 & 14\\
3 & 3 & 7 & 10 & 14 & 14 & 14 & 14 & 14\\
4 & 2 & 5 & 8 & 11 & 14 & 14 & 14 & 14\\
5 & 2 & 4 & 7 & 9 & 11 & 14 & 14 & 14\\
6 & 2 & 4 & 6 & 8 & 10 & 12 & 14 & 14\\
7 & 1 & 3 & 5 & 7 & 8 & 10 & 12 & 14\\
\end{tabular}
\captionof{table}{The quantities $\gp^k_d(C_{14})$ for various values of $k$ and $d$.}\label{sophisticatedtable2}
\end{center}

\section{Infinite 2D-grids}\label{section_infinite_2d_grids}

Klav\v{z}ar and Manuel \cite{MR3879620} used a corollary of the famous theorem of Erd\H{o}s and Szerkeres to obtain a lower bound on $\gp(P_\infty^{\Box,2})$. We do the same to obtain a lower bound on $\gp^k(P_\infty^{\Box,2})$. Recall that $\gp^k(P_\infty^{\Box,2})$ is the largest cardinality of a set of vertices $A$ in $\gp^k(P_\infty^{\Box,2})$ with no $k$ vertices in a common geodesic.

\begin{theorem}[Erd\H{o}s and Szekeres]
    For integers $r,s\geq 2$, any sequence of distinct real numbers with length $(r-1)(s-1)+1$ must contain a monotonically increasing subsequence of length $r$ or a monotonically decreasing subsequence of length $s$.
\end{theorem}

\begin{corollary}
For every integer $n\geq 2$, every sequence of real numbers $(a_1,\ldots,a_N)$ with length $N\geq (n-1)^2+1$ contains a monotone subsequence of length $n$.
\end{corollary}

\begin{corollary}\label{corollary_ES}
    If $n\geq 1$ is an integer and $S\subseteq \R^2$ with $|S|\geq (n-1)^2+1$, then $S$ contains $n$ points that form a monotone sequence.
\end{corollary}

We adapt Klav\v{z}ar and Manuel's notion of monotone-geodesic labeling to that of $k$-monotone geodesic labeling. Here $3$-monotone geodesic labeling is equivalent to monotone-geodesic labeling.

\begin{definition}\label{definition_k_mon_geo_lab}
    Let $G$ be a graph. An injective function $f:V(G)\to\R^2$ is called a \emph{$k$-monotone-geodesic labeling} of $G$ if whenever $v_1,\ldots,v_k\in V(G)$ and the sequence $(f(v_1),\ldots, f(v_k))$ is monotone, then there is a geodesic $g\in\G(G)$ with $v_1,\ldots,v_k\in V(g)$.
\end{definition}

Now we connect $k$-monotone-geodesic labelings with $k$-general position sets.

\begin{lemma}\label{lemma_k_mon_geodesic}
    If a graph $G$ admits a $k$-monotone-geodesic labeling for some integer $k\geq 2$, then $\gp^k(G)\geq (k-1)^2$.
\end{lemma}
\begin{proof}
    Suppose $f:V(G)\to\R^2$ is a $k$-monotone-geodesic labeling and $S\subseteq V(G)$ is a $k$-general position set of $G$ with $|S|\geq (k-1)^2+1$. Then, by Corollary \ref{corollary_ES}, the set $f(S)\subseteq\R^2$ has a subset $T$ of size $k$ which forms a monotone sequence in $\R^2$. Then $f^{-1}(T)$ has cardinality $k$ and is contained in a single geodesic of $G$, a contradiction.
\end{proof}

We now obtain the main result of this section. A graph $G$ is called a \emph{grid graph} if it is an induced connected subgraph of $P^{\Box, 2}_\infty$.

\begin{theorem}
    Suppose $G$ is a grid graph and $f:V(P_\infty^{\Box,2})\to\R^2$ is the natural labeling of $P_\infty^{\Box,2}$. For an integer $k\geq 2$, if $G$ contains $P_{2k-3}\cprod P_{2k-3}$ as a subgraph and $f\restrict G$ is a $k$-monotone-geodesic labeling, then $\gp^k(G)=(k-1)^2$.
\end{theorem}

\begin{proof}
    Since $f\restrict G$ is a $k$-monotone-geodesic labeling, Lemma \ref{lemma_k_mon_geodesic} implies that $\gp^k(G)\geq (k-1)^2$. 
    Given a positive integer $r$, we define
    \[A_r=\left\{(x,y)\in P_\infty^{\Box,2}\st d((0,0),(x,y))\in [0,r]\text{ and }d((0,0),(x,y))\equiv r \bmod 2\right\}.\]
    For $(a,b)\in P_\infty^{\Box,2}$ we let 
    \[A_r(a,b)=\left\{(a,b)+(x,y)\st (x,y)\in A_r\right\}.\]
    The set $A_{k-2}\left(\ceil{\frac{2k-3}{2}},\ceil{\frac{2k-3}{2}}\right)$ has cardinality $(k-1)^2$ and is a $k$-general position set of $P_{2k-3}\cprod P_{2k-3}$. Since $P_{2k-3}\cprod P_{2k-3}$ is an isomoetric subgraph of $G$, it follows that $A_{k-2}\left(\ceil{\frac{2k-3}{2}},\ceil{\frac{2k-3}{2}}\right)$ is a $k$-general position set of $G$.
\end{proof}

Now we obtain an easy corollary, which was stated in Section \ref{section_intro}.

\begin{corollary}\label{corollary_infinite_grid} If $k\geq 2$ is an integer, then
    \[\gp^k(P_\infty^{\Box,2})=(k-1)^2.\]
\end{corollary}

\section{Thin finite grids}\label{section_thin_grids}

What would we need to do in order to find $\gp^k_d(P_n \cprod P_m)$? If a set $S$ of vertices in $P_n\cprod P_m$ has size $|S|\geq (k-1)^2+1$, then, by the results of the previous section, this implies that $S$ has a monotone subset of size $k$, and then this subset is contained in a geodesic. But, how large must $S$ be to guarantee that $S$ has a monotone subset $M\subseteq S$ of size $k$ such that $\max\{d(u,v)\st u,v\in M\}\leq d$? In this section we answer this question when $m=2$.

For integers $d$ and $k$ with $1\leq k-1\leq d$, we will define subsets $A_{d,k}$ and $B_{d,k}$ of $\Z_{\geq 1}\times\{1,2\}$ that will allow us to easily define $k$-general $d$-position subsets of grids of the form $P_n\cprod P_2$. In what follows, we will see that for every $n\geq 1$, whenever they are defined, the sets $A_{k,d}\cap V(P_n\cprod P_2)$ and $B_{k,d}\cap V(P_n\cprod P_2)$ are $k$-general $d$-position in $P_n\cprod P_2$, and furthermore, one of the sets $A_{k,d}\cap V(P_n\cprod P_2)$ and $B_{k,d}\cap V(P_n\cprod P_2)$, must have the largest possible cardinality of a $k$-general $d$-position set in $P_n\cprod P_2$.

\begin{definition}\label{definition_Akd_Bkd}
Suppose $d$ and $k$ are integers with $d\geq 1$ and $k\geq 2$. If $d\geq 2k-3$, we define a subset $A_{k,d}$ of $\Z_{\geq 1}\times\{1,2\}$ as follows. For each integer $s\geq 0$, let
\[A_{k,d}(s)=\left\{\left(i, 1+\frac{1+(-1)^{i+s}}{2}\right)\st i\in [ds+1,ds+2k-3] \right\}.\] 
Define 
\[A_{k,d}=\bigcup_{s=0}^\infty A_{k,d}(s)\]
(see Figure \ref{figure_A_5_9}).
For $d\geq k-2$ we define a subset $B_{k,d}$ of $\Z_{\geq 1}\times\{1,2\}$ as follows. For each integer $s$, let
\[B_{k,d}(s)=\left[ds+1,ds+k-2\right]\times\{1,2\}.\]
Define
\[B_{k,d}=\bigcup_{s=0}^\infty B_{k,d}(s)\]
(see Figure \ref{figure_B_5_9}).

\begin{figure}
\centering
\begin{tikzpicture}[scale=0.5]
  \def\n{22} 
  \def\m{2}  
  \def\d{9}
  \def\q{\numexpr\n/\d}
  \def\r{mod(\n,\d)}
  \def\ka{5}
  
  \foreach \i in {1,...,\n} {
    \foreach \j in {1,...,\m} {
    \ifthenelse{\i<\n}{
        \draw (\i, \j) -- (\i+1, \j);
      }
      {
        \draw (\i, \j) -- (\i+0.5, \j);
        \node[] () at (\i+1.2,1.5) {$\cdots$};
      };  
    }
    }

  \foreach \j in {1,...,\numexpr\m-1} {
    \foreach \i in {1,...,\n} {
      \draw (\i, \j) -- (\i, \numexpr\j+1);
    }
  }
  
  \foreach \i in {1,...,\n} {
    \foreach \j in {1,...,\m} {
      \fill[white] (\i, \j) circle (3pt);
      \draw (\i, \j) circle (3pt); 
    }
  }



\foreach \i in {1,...,\n} {
    \node[] () at (\i,0) {\tiny $\i$};
}

\foreach \j in {1,...,\m} {
    \node[] () at (0,\j) {\tiny $\j$};
}


\foreach \s in {0,...,\q} {
    \ifthenelse{\s=\q}
    {
        \pgfmathparse{\r}
        \foreach \i in {1,...,\pgfmathresult} {
            \fill[black] ({\i+\s*\d}, {1+(1+(-1)^(\i+\s))/2}) circle (3pt);
        }
    }
    {
        \foreach \i in {1,...,\numexpr\ka*2-3} {
            \fill[black] ({\i+\s*\d}, {1+(1+(-1)^(\i+\s))/2}) circle (3pt);
        }
    }
}

\end{tikzpicture}
\caption{The set $A_{5,9}\subseteq\Z_{\geq 1}\times\{1,2\}$.}
\label{figure_A_5_9}
\end{figure}
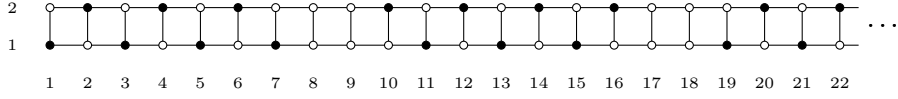

\begin{figure}
\centering
\begin{tikzpicture}[scale=0.5]
  \def\n{22} 
  \def\m{2}  
  \def\d{9}
  \def\q{\numexpr\n/\d}
  \def\r{mod(\n,\d)}
  \def\ka{5}
  
  \foreach \i in {1,...,\n} {
    \foreach \j in {1,...,\m} {
    \ifthenelse{\i<\n}{
        \draw (\i, \j) -- (\i+1, \j);
      }
      {
        \draw (\i, \j) -- (\i+0.5, \j);
        \node[] () at (\i+1.2,1.5) {$\cdots$};
      };  
    }
    }

  \foreach \j in {1,...,\numexpr\m-1} {
    \foreach \i in {1,...,\n} {
      \draw (\i, \j) -- (\i, \numexpr\j+1);
    }
  }
  
  \foreach \i in {1,...,\n} {
    \foreach \j in {1,...,\m} {
      \fill[white] (\i, \j) circle (3pt);
      \draw (\i, \j) circle (3pt); 
    }
  }



\foreach \i in {1,...,\n} {
    \node[] () at (\i,0) {\tiny $\i$};
}

\foreach \j in {1,...,\m} {
    \node[] () at (0,\j) {\tiny $\j$};
}


\foreach \s in {0,...,\q} {
    \ifthenelse{\s=\q}
    {
        \pgfmathparse{\r}
        \foreach \i in {1,...,3} {
            \fill[black] ({\i+\s*\d}, 1) circle (3pt);
            \fill[black] ({\i+\s*\d}, 2) circle (3pt);
        }
    }
    {
        \foreach \i in {1,...,\numexpr\ka-2} {
            \fill[black] ({\i+\s*\d}, 1) circle (3pt);
            \fill[black] ({\i+\s*\d}, 2) circle (3pt);
        }
    }
}

\end{tikzpicture}
\caption{The set $B_{5,9}\subseteq\Z_{\geq 1}\times\{1,2\}$.}
\label{figure_B_5_9}
\end{figure}
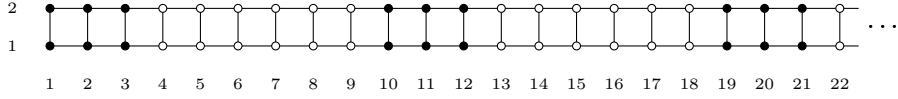

\end{definition}

First we show that $A_{k,d}$ and $B_{k,d}$ yield $k$-general $d$-position sets in grids of the form $P_n\cprod P_2$, and we take note of the cardinalities of these $k$-general $d$-position sets.

\begin{lemma}\label{lemma_Akd_and_Bkd}
Suppose $k$ and $d$ are integers with $d\geq 1$ and $k\geq 2$. 
\begin{enumerate}
    \item Suppose $d\geq 2k-3$. Then $A_{k,d}$ is a $k$-general $d$-position set in $P_{\infty}\cprod P_2$. Furthermore, for every integer $n\geq 1$, $A_{k,d}\cap V(P_n\cprod P_2)$ is a $k$-general $d$-position set in $P_n\cprod P_2$ with
    \[|A_{k,d}\cap V(P_n\cprod P_2)|=(2k-3)\floor{\frac{n}{d}}+\min(n\bmod d,2k-3).\]
    \item Suppose $d\geq k-2$. Then $B_{k,d}$ is a $k$-general $d$-position set in $P_{\infty}\cprod P_2$. Furthermore, for every integer $n\geq 1$, $B_{k,d}\cap V(P_n\cprod P_2)$ is a $k$-general $d$-position set in $P_\infty\cprod P_2$ with
    \[|B_{k,d}\cap V(P_n\cprod P_2)|=2(k-2)\floor{\frac{n}{d}}+2\min(n\bmod d,k-2).\]
\end{enumerate}
\end{lemma}

\begin{proof}
Let us prove that, when they are defined $A_{k,d}$ and $B_{k,d}$ are $k$-general $d$-position sets in $P_\infty\cprod P_2$. The rest follows easily from Lemma \ref{lemma_all_vertices} as well as the definitions of $A_{k,d}$ and $B_{k,d}$. When $A_{k,d}$ and $B_{k,d}$ are defined, for each $s\geq 0$ we have $|A_{k,d}(s)|=2k-3$ and $|B_{k,d}(s)|=2(k-2)=2k-4$. For any geodesic $g$ of $P_\infty\cprod P_2$, it follows that for every $s\geq 0$ we have $|A_{k,d}(s)\cap V(g)|\leq \ceil{\frac{2k-3}{2}}=k-1$ and $|B_{k,d}(s)\cap V(g)|\leq k-1$. Suppose $s$ and $s'$ are integers such that $s,s'\geq 0$ and $|s-s'|\leq 1$. Suppose $g$ is a geodesic of $P_\infty\cprod P_2$ with length $\lambda(g)=d$. If $|A_{k,d}(s)\cap V(g)|=a\leq k-1$ then $|A_{k,d}(s')\cap V(g)|=k-1-a$. Thus, it follows that $|A_{k,d}\cap V(g)|\leq k-1$. Similarly $|B_{k,d}\cap V(g)|\leq k-1$. Therefore $A_{k,d}$ and $B_{k,d}$ are $k$-general $d$-position sets in $P_{\infty}\cprod P_2$.
\end{proof}

First we obtain the desired result in the case where $k=2$.

\begin{theorem}\label{theorem_when_k_is_2}
    Suppose $d$ and $n$ are integers such that $d\geq 1$ and $n\geq 1$. Then
    \[\gp^2_d(P_n\cprod P_2)=\ceil{\frac{n}{d}}.\]    
\end{theorem}

\begin{proof}
Since $d\geq 2(2)-3=1$ the set $A_{2,d}$ is defined, and by Lemma \ref{lemma_Akd_and_Bkd}, $A_{2,d}\cap V(P_n\cprod P_2)$ is is $2$-general $d$-position in $P_n\cprod P_2$. Also by Lemma \ref{lemma_Akd_and_Bkd},
\[|A_{2,d}\cap V(P_n\cprod P_2)|=\floor{\frac{n}{d}}+\min(n\bmod d,1)=\ceil{\frac{n}{d}}.\]
Thus $\gp^2_d(P_n\cprod P_2)\geq \ceil{\frac{n}{d}}$.

Suppose $A$ is a $2$-general $d$-position set in $P_n\cprod P_2$ with $|A|>\ceil{\frac{n}{d}}$. Let $q=\floor{\frac{n}{d}}$ and $r=n\bmod d$. Then $n=qd+r$ where $0\leq r<d$. For each $i\in\{0,\ldots,q-1\}$ let $X_i=[di+1,d(i+1)]\times\{1,2\}$, and if $r>0$ let $X_q=[dq+1,dq+r]\times\{1,2\}$. Since $|A|>\ceil{\frac{n}{d}}$ at least one part of the partition defined by the $X_i$'s must contain at least two vertices of $A$. But the diameter of each induced isometric subgraph $(P_n\cprod P_2)[X_i]$ is at most $d$, which implies that $A$ is not $2$-general $d$-position in $P_n\cprod P_2$.
\end{proof}

Next we obtain the desired result in the case where $k=3$.

\begin{theorem}\label{theorem_when_k_is_3}
    Suppose $d$ and $n$ are integers such that $d\geq 1$ and $n\geq 1$. 
    \begin{enumerate}
        \item If $n\leq d$ then
        \[\gp^3_d(P_n\cprod P_2)=\begin{cases}
        2 &\text{if $n=1$}\\
        \min(n,3) &\text{if $n\geq 2$.}\\
        \end{cases}\]
        \item If $n\geq d+1$ then
        \[\gp^3_d(P_n\cprod P_2)=\begin{cases}
        \displaystyle 2\floor{\frac{n}{2}}+2\min(n\bmod 2,1) &\text{if $d=2$}\\[10pt]
        \displaystyle 3\floor{\frac{n}{d}}+\min(n\bmod d,3)&\text{if $d\geq 3$.}
        \end{cases}\]
    \end{enumerate}
\end{theorem}

\begin{proof}
Let us prove (1). Suppose $1=n\leq d$. In $P_1\cprod P_2$, there are no geodesics with length greater than $d\geq 1$, so $\gp^3_d(P_1\cprod P_2)=\gp^3(P_1\cprod P_2)=\gp^3(P_2)=2$. Suppose $n=2$. Similarly, in $P_2\cprod P_2$, there are no geodesics with length greater than $d\geq 2$ and hence $\gp^3_d(P_2\cprod P_2)=\gp^3(P_2\cprod P_2)=2$.

Suppose $n\leq d$ and $n\geq 3$. Then, $A_{3,d}$ is defined (because $d\geq 2(3)-3$) and by Lemma \ref{lemma_Akd_and_Bkd}, $A_{3,d}\cap V(P_n\cprod P_2)$ is a $3$-general $d$-position set in $P_n\cprod P_2$ with 
\[|A_{3,d}\cap V(P_n\cprod P_2)|=3\floor{\frac{n}{d}}+\min(n\bmod d,3)=3.\]
Thus, $\gp^3_d(P_n\cprod P_2)\geq 3$. Let us prove that $\gp^3_d(P_n\cprod P_2)\leq 3$. For the sake of contradiction suppose $A$ is a $k$-general $d$-position set in $P_n\cprod P_2$ with $|A|>3$. Without loss of generality we can assume that $|A|=4$. Since $[1,n]\times\{1\}$ and $[1,n]\times\{2\}$ are vertex sets of geodesics of length less than $d$, it follows that $|A\cap ([1,n]\times\{1\})|=2$ and $|A\cap ([1,n]\times\{2\})|=2$. Let $i_0$ be the least $i\in[1,n]$ such that $A\cap \{(i,1),(i,2)\}\neq\emptyset$. Without loss of generality, suppose $(i_0,1)\in A$. Then $\{(i_0,1)\}\cup ([i_0,n]\times\{2\})$ is the vertex set of a geodesic of $P_n\cprod P_2$ with length at most $d$ containing $3$ vertices from $A$, a contradiction. This establishes (1).

Now let us prove (2). For the remainder of the proof, let $q=\floor{\frac{n}{d}}$ and  $r=n\bmod d$. Suppose $n\geq d+1$ and $d=2$. Notice that $r\in\{0,1\}$. By Lemma \ref{lemma_Akd_and_Bkd}, $B_{3,2}\cap V(P_n\cprod p_2)$ is a $3$-general $2$-position set in $P_n\cprod P_2$ with 
\[|B_{3,2}\cap V(P_n\cprod p_2)|=2\floor{\frac{n}{2}}+2\min(r,1).\]
Thus $\gp^3_2(P_n\cprod P_2)\geq 2\floor{\frac{n}{2}}+2\min(r,1)$. To show that $\gp^3_2(P_n\cprod P_2)\leq 2\floor{\frac{n}{2}}+2\min(r,1)$ we will use Lemma \ref{lemma_isometric_upper_bound}. For $i\in\{0,\ldots,q-1\}$, let $X_i=[2i+1,2i+2]\times\{1,2\}$. If $r=1$, let $X_q=\{2q+1\}\times\{1,2\}$ and if $r=0$ let $X_q=\emptyset$. For $i\in\{0,\ldots,q\}$, let $H_i=(P_n\cprod P_2)[X_i]$ and note that $H_i$ is an isometric subgraph of $P_n\cprod P_2$. Furthermore, $V(P_n\cprod P_2)=\bigcup_{i=0}^qV(H_i)$. Hence, by Lemma \ref{lemma_isometric_upper_bound}, and by using the fact that $H_i\cong P_2\cprod P_2$ for $i\in\{0,\ldots,q-1\}$ and $H_q\cong P_r\cprod P_2$, we obtain
\[\gp^3_2(P_n\cprod P_2)\leq \sum_{i=0}^q\gp^3_2(H_i)=\gp^3_2(P_2\cprod P_2)q+\gp^3_2(P_r\cprod P_2)=2q+2m=2\ceil{\frac{n}{2}}.\]

Suppose $n\geq d+1$ and $d\geq 3$. By Lemma \ref{lemma_Akd_and_Bkd}, $A_{3,d}$ is a $3$-general $d$-position set in $P_n\cprod P_2$ with 
\[|A_{3,d}\cap V(P_n\cprod P_2)|=3\floor{\frac{n}{d}}+\min(r,3).\]
Thus $\gp^3_d(P_n\cprod P_2)\leq 3\floor{\frac{n}{d}}+\min(r,3)$. Now we will prove the reverse direction of this inequality. 

For the remainder of the proof, let us fix some notation. For $i\in\{0,\ldots,q-1\}$, let $X_i=[di+1,di+d]\times\{1,2\}$. If $r>0$, let $X_q=[dq+1,dq+r]\times\{1,2\}$ and if $r=0$, let $X_q=\emptyset$. For $i\in\{0,\ldots,q-1\}$ let $H_i=(P_n\cprod P_2)[X_i]$.

First, suppose $r\neq 1$. By Lemma \ref{lemma_isometric_upper_bound}, it follows from (1) that
\begin{align*}
    \gp^3_d(P_n\cprod P_2)\leq \sum_{i=0}^q\gp^3_d(H_i)&=\gp^3_d(P_d\cprod P_2)q+\gp^3_d(P_r\cprod P_2)\\
        &=\min(d,3)q+\min(r,3)\\
        &=3q+m,
\end{align*}
as desired.

Suppose $r=1$. Then $V(H_q)=X_q=\{dq+1\}\times\{1,2\}$ contains two vertices. We must show that $\gp^3_d(P_n\cprod P_2)\leq 3q+1$. For the sake of contradiction, suppose $A$ is a $3$-general $d$-position set in $P_n\cprod P_2$ with $|A|>3q+1.$ Without loss of generality suppose $|A|=3q+2$. For $i\in\{0,\ldots,q-1\}$ we have $\gp^3_d(H_i)=\gp^3_d(P_d\cprod P_2)=3$ and thus, by Lemma \ref{lemma_isometric_down_convex_down_and_up}, $|A\cap X_i|=3$. This implies that $A\cap X_q=2$. By the pigeonhole principle, we can let $j\in\{1,2\}$ be such that $[d(q-1)+1,dq]\times\{j\}$ contains two vertices of $A$. But then $[d(q-1)+1,dq+1]\times\{j\}$ is the vertex set of a geodesic of length $d$ that contains three vertices of $A$, a contradiction.
\end{proof}

Before we obtain the value of $\gp^k_d(P_n\cprod P_2)$ for $k\geq 4$, we need a few lemmas.

\begin{lemma}\label{lemma_any}
Suppose $d$ and $k$ are positive integers with $k\geq 4$ and $d\geq 2k-3$. Let $A$ be any $k$-general $d$-position set in $P_d\cprod P_2$ with $|A|=2k-3$. For all $i\in[1,d]$ we have that
\begin{enumerate}
    \item if $j\in\{1,2\}$ is such that $|A\cap ([1,d]\times\{j\})|=k-1$, then
\[|A\cap ([1,i]\times\{j\})|\leq\ceil{\frac{i}{2}}\]
and
\item if $j\in\{1,2\}$ is such that $|A\cap ([1,d]\times\{j\})|=k-2$, then 
\[|A\cap ([1,i]\times\{j\})|\leq\floor{\frac{i}{2}}.\]
\end{enumerate}
\end{lemma}

\begin{proof}
Without loss of generality, suppose $|A\cap([1,d]\times\{1\})|=k-1$ and $|A\cap([1,d]\times\{2\})|=k-2$. We prove that (1) and (2) hold for all desired values of $i$ by simultaneous induction.

Notice that $(1,2)\notin A$ because if $(1,2)\in A$ then $\{(1,2)\}\cup ([1,d]\times\{1\})$ would be the vertex set of a geodesic of $P_d\cprod P_2$ of length $d$ containing $k$ vertices of $A$, a contradiction. This implies that $|A\cap([1,1]\times\{2\})|\leq\floor{\frac{1}{2}}=0$ and $|A\cap([1,2]\times\{2\})|\leq\floor{\frac{2}{2}}=1$.

We must have $|A\cap ([1,2]\times\{1\})|\leq 1$ because otherwise, if $|A\cap ([1,2]\times\{1\})|=2$ then, since $(1,2)\notin A$ by the previous case, it follows that 
\[([1,2]\times\{1\})\cup ([2,d]\times\{2\})\]
is the vertex set of a geodesic of $P_d\cprod P_2$ containing $k$ vertices of $A$, a contradiction. This implies that $|A\cap ([1,1]\times\{1\})|\leq\ceil{\frac{1}{2}}=1$ and $|A\cap ([1,2]\times\{1\})|\leq\ceil{\frac{2}{2}}=1$.

Suppose (1) and (2) hold for $i$. Since $|A\cap ([1,i]\times\{1\})|\leq \ceil{\frac{i}{2}}$, we must have $|A\cap([1,i+1]\times\{2\})|\leq\floor{\frac{i+1}{2}}$ because otherwise, if $|A\cap([1,i+1]\times\{2\})|>\floor{\frac{i+1}{2}}$, then
there is a geodesic $g$ of $P_d\cprod P_2$ with
\[V(g)=([1,i+1]\times\{2\})\cup([i+1,d]\times\{1\})\]
such that $V(g)$ contains at least
\[k-1-\ceil{\frac{i}{2}}+\floor{\frac{i+1}{2}}+1=k\]
vertices of $A$, a contradiction. 

Similarly, since $|A\cap([1,i]\times\{2\})|\leq\floor{\frac{i}{2}}$, we must have $|A\cap([1,i+1]\times\{1\})|\leq\ceil{\frac{i+1}{2}}$ because otherwise, if $|A\cap([1,i+1]\times\{1\})|\geq\ceil{\frac{i+1}{2}}$ then there is a geodesic $g$ of $P_d\cprod P-2$ with
\[V(g)=([1,i+1]\times\{1\})\cup([i+1,d]\times\{2\})\]
such that $V(g)$ contains at least
\[k-2-\floor{\frac{i}{2}}+\ceil{\frac{i+1}{2}}+1=k\]
vertices of $A$, a contradiction.
\end{proof}

\begin{lemma}\label{lemma_not_kgdp}
Suppose $d$, $k$ and $r$ are positive integers such that $k\geq 4$, $r\geq 1$ and $d\geq 2k-3$. If $A$ is any set of vertices in $P_{d+r}\cprod P_2$ with $|A\cap ([1,d]\times \{1,2\})|\geq 2k-3$ and $|A\cap([d+1,d+r]\times\{1,2\})|\geq r+1$, then $A$ is not a $k$-general $d$-position set in $P_{d+r}\cprod P_2$.
\end{lemma}

\begin{proof}
For the sake of contradiction, suppose $A$ is a $k$-general $d$-position set in $P_{d+r}\cprod P_2$ satisfying the stated hypotheses.

Without loss of generality suppose $|A\cap ([1,d]\times\{1\})|\geq k-1$ and $|A\cap ([1,d]\times\{2\})|\geq k-2$. It follows by the pigeonhole principle that either 
\[|A\cap ([d+1,d+r]\times\{1\})|\geq\ceil{\frac{r+1}{2}}\]
or 
\[|A\cap ([d+1,d+r]\times\{2\})|\geq\ceil{\frac{r+1}{2}}.\] 
Suppose $|A\cap ([d+1,d+r]\times\{1\})|\geq\ceil{\frac{r+1}{2}}$. Then there is a geodeisc $g$ of $P_{d+r}\cprod P_2$ with
\[V(g)=[r,d+r]\times\{1\}\]
which has length $d$. By Lemma \ref{lemma_any}, the set $[1,r-1]\times\{1\}$ contains at most $\ceil{\frac{r-1}{2}}$ vertices of $A$. Therefore, $g$ contains at least
\[k-1-\ceil{\frac{r-1}{2}}+\ceil{\frac{r+1}{2}}=k\]
vertices of $A$, a contradiction.

Now suppose $|A\cap ([d+1,d+r]\times\{2\})|\geq\ceil{\frac{r+1}{2}}$. We can assume, without loss of generality, that $r$ is even because when $r$ is odd we have $\ceil{\frac{r+1}{2}}=\floor{\frac{r+1}{2}}$, and thus the assumption $|A\cap ([d+1,d+r]\times\{1\})|\geq\ceil{\frac{r+1}{2}}$ of the previous case holds. By Lemma \ref{lemma_any}, the set $[1,r-1]\times\{2\}$ contains at most $\floor{\frac{r-1}{2}}$ vertices of $A$. Thus, there is a geodesic $g'$ of $P_{d+r}\cprod P_2$ with
\[V(g')=[r,d+r]\times\{2\}\]
such that $V(g')$ contains at least
\[k-2+\floor{\frac{r-1}{2}}+\ceil{\frac{r+1}{2}}=k\]
vertices of $A$, where the last equality holds because $r$ is even. Since $g'$ has length $d$ we obtain a contradiction.
\end{proof}

Now we obtain the values of $\gp^k_d(P_n\cprod P_2)$ for $k\geq 4$. First, let us state the values for $\gp^k_d(P_n\,\Box\, P_2)$ for $k\geq 4$ in terms of the sets $A_{k,d}$ and $B_{k,d}$ (see Definition \ref{definition_Akd_Bkd}).

\begin{theorem}\label{theorem_thin_grids_Akd_Bkd}
Suppose $d$, $k$ and $n$ are positive integers such that $k\geq 4$ and $d\geq k-1$.
\begin{enumerate}
    \item If $n\leq d$ then
    \[\gp^k_d(P_n\,\Box\, P_2)=\begin{cases}
        |B_{k,d}\cap V(P_n\,\Box\, P_2)| & \text{if $1\leq n<2k-3$}\\
        |A_{k,d}\cap V(P_n\,\Box\, P_2)| & \text{if $n\geq 2k-3$.}
    \end{cases}\]
    \item If $n\geq d+1$ and $d<2k-3$ then
    \[\gp^k_d(P_n\,\Box\, P_2)=|B_{k,d}\cap V(P_n\,\Box\, P_2)|.\]
    \item If $n\geq d+1$ and $d\geq 2k-3$ then
    \begin{align*}
        \gp^k_d(P_n\,\Box\, P_2)=\max\left(|A_{k,d}\cap V(P_n\,\Box\, P_2)|, |B_{k,d}\cap V(P_n\,\Box\, P_2)|\right)
    \end{align*}
\end{enumerate}
\end{theorem}

By using the formulas for the cardinalities of $A_{k,d}\cap V(P_n\cprod P_2)$ and $A_{k,d}\cap V(P_n\cprod P_2)$, and by considering cases, it is easy to see that Theorem \ref{theorem_thin_grids_Akd_Bkd} is equivalent to the next theorem; we restate Theorem \ref{theorem_thin_grids_Akd_Bkd} in this way because this will facilitate our proof.

\begin{theorem}\label{theorem_gpkd_grids_part1}
Suppose $d$, $k$ and $n$ are positive integers such that $k\geq 4$ and $d\geq k-1$. 
\begin{enumerate}
\item\label{item_short_grids} If $n\leq d$ then
\begin{align*}\gp^k_d(P_n\cprod P_2)=\begin{cases}
    2\min(n,k-2) & \text{if $1\leq n\leq 2k-4$}\\
    2k-3 & \text{if $n\geq 2k-3$}\\
\end{cases}.\end{align*}
\item If $n\geq d+1$ and $d<2k-3$ then
\[\gp^k_d(P_n\cprod P_2)=(2k-4)\floor{\frac{n}{d}}+2\min(n\bmod d,k-2).\]
\item Suppose $n\geq d+1$ and $d\geq 2k-3$.
\begin{enumerate}
    \item Suppose $n\bmod d\leq k-2$.
    \begin{enumerate}
        \item If $n\bmod d \leq \floor{\frac{n}{d}}$ then \[\gp^k_d(P_n\cprod P_2)=(2k-3)\floor{\frac{n}{d}}+n\bmod d.\]
        \item If $n\bmod d>\floor{\frac{n}{d}}$ then \[\gp^k_d(P_n\cprod P_2)=(2k-4)\floor{\frac{n}{d}}+2(n\bmod d).\]
    \end{enumerate}
    \item Suppose $k-2 < n\bmod d <2k-3$.
    \begin{enumerate}
        \item If $2k-4- n\bmod d\leq \floor{\frac{n}{d}}$ then \[\gp^k_d(P_n\cprod P_2)=(2k-3)\floor{\frac{n}{d}}+n\bmod d.\]
        \item If $2k-4-n\bmod d>\floor{\frac{n}{d}}$ then \[\gp^k_d(P_n\cprod P_2)=(2k-4)\left(\floor{\frac{n}{d}}+1\right).\]
    \end{enumerate}
    \item If $n\bmod d\geq 2k-3$ then 
    \[\gp^k_d(P_n\cprod P_2)=(2k-3)\left(\floor{\frac{n}{d}}+1\right).\]
\end{enumerate}
\end{enumerate}
\end{theorem}

\begin{proof}

For (1), suppose $n\leq d$. To prove the formula in (1), let us consider a few cases. The largest cardinality of a vertex set of a geodesic in $P_n\cprod P_2$ is $n+1$, so if $1\leq n\leq k-2$, there are no geodesics of $P_n\cprod P_2$ that contain $k$ vertices of $V(P_n \cprod P_2)$, which implies that $A_n=V(P_n\cprod P_2)$ is a $k$-general $d$-position set in $P_n\cprod P_2$ with $|A_n|=2n$. Therefore, when $1\leq n\leq k-2$ we have $\gp^k_d(P_n\cprod P_2)=2n$.

Suppose $k-1\leq n\leq 2k-4$. By Lemma \ref{lemma_Akd_and_Bkd}, $B_{k,d}\cap V(P_n\cprod P_2)$ is a $k$-general $d$-position set in $P_n\cprod P_2$ with cardinality
\[|B_{k,d}\cap V(P_n\cprod P_2)|=2(k-2)\floor{\frac{n}{d}}+2\min(n\bmod d,k-2)=2(k-2)\]
(where the last equality can be seen to hold by considering cases in which $n<d$ and $n=d$). Thus we have $\gp^k(P_n\cprod P_2)\geq 2k-4$. To show that $\gp^k(P_n\cprod P_2)\leq 2k-4$, assume for the sake contradiction that $A$ is a $k$-general $d$-position set of vertices in $P_n\cprod P_2$ with $|A|>2k-4$. Since $n\leq 2(k-2)$, it follows by the pigeonhole principle that we may fix an $i\in[1,n]$ such that $|A\cap (\{i\}\times\{1,2\})|=2$. There are exactly two geodesics, $g_1$ and $g_2$ of $P_n\cprod P_2$ whose vertex sets contain $\{i\}\times\{1,2\}$ and that have maximal length $\lambda(g_1)=\lambda(g_2)=n$. Since $|A\setminus(\{i\}\times\{1,2\})|>2k-6$, and since all of the vertices in $A\setminus(\{i\}\times\{1,2\})$ must live in $g_1$ or $g_2$, it follows by the pigeonhole principle that for some $\ell\in\{1,2\}$ we have 
\[|(A\setminus(\{i\}\times\{1,2\}))\cap V(g_\ell)|\geq k-2\]
and hence 
\[|A\cap V(g_\ell)|\geq k,\]
where $\lambda(g_\ell)=n\leq d$, a contradiction.

To finish (1), suppose $2k-3\leq n\leq d$. We must show that $\gp^k_d(P_n\cprod P_2)=2k-3$. By Lemma \ref{lemma_Akd_and_Bkd}, $A_{k,d}\cap V(P_n\cprod P_2)$ is a $k$-general $d$-positions set in $P_n\cprod P_2$ with cardinality
\[|A_{k,d}\cap V(P_n\cprod P_2)|=(2k-3)\floor{\frac{n}{d}}+\min(n\bmod d,2k-3)=2k-3\]
(where the last inequality can be seen to hold by considering cases in which $n<d$ and $n=d$).
Thus, it follows that $\gp^k(P_n\cprod P_2)\geq 2k-3$. To show that $\gp^k(P_n\cprod P_2)\leq 2k-3$, assume for the sake of contradiction that $A$ is a $k$-general $d$-position set of vertices in $P_n\cprod P_2$ with $|A|>2k-3$. Without loss of generality we can assume that $|A|=2(k-1)$. If $|A\cap ([1,n]\times\{j\})|\geq k$ for some $j\in\{1,2\}$, then $[1,n]\times\{j\}$ is the vertex set of a geodesic containing at least $k$ elements of $A$ that has length $n-1<d$, a contradiction. Thus, we can assume that
\[|A\cap([1,n]\times\{j\})|=k-1\] 
for $j\in\{1,2\}$. Let $i_0$ be the least $i\in[1,n]$ such that $A\cap (\{i\}\times\{1,2\})\neq\emptyset$. Without loss of generality, suppose $(i_0,1)\in A$. Then it follows that there is a geodeisc $g$ with 
\[V(g)=\{(i_0,1),(i_0,2)\}\cup\{(i,2)\st i_0\leq i\leq n\},\]
which contains $k$ members of $A$ and has length at most $n\leq d$, a contradiction.

For (2), suppose $n\geq d+1$ and $d<2k-3$. By Lemma \ref{lemma_Akd_and_Bkd}, $B_{k,d}\cap V(P_n\cprod P_2)$ is a $k$-general $d$-position set in $P_n\cprod P_2$ with cardinality
\[|B_{k,d}\cap V(P_n\cprod P_2)|=2(k-2)\floor{\frac{n}{2}}+2\min(n\bmod d,k-2)\]
and thus
\[\gp^k_d(P_n\cprod P_2)\geq 2(k-2)\floor{\frac{n}{2}}+2\min(n\bmod d,k-2).\]
Let us show that 
\[\gp^k_d(P_n\cprod P_2)\leq 2(k-2)\floor{\frac{n}{2}}+2\min(n\bmod d,k-2).\]
Let $q=\floor{\frac{n}{d}}$ and let $r= n\bmod d$ be the unique integer such that $n=dq+r$ where $0\leq r<d$.
For each $i\in\{0,\ldots,q-1\}$ let $X_i=[di+1,di+d]\times\{1,2\}$. If $r=0$ let $X_i=\emptyset$ and if $r>0$ let $X_i=[dq+1,n]\times\{1,2\}$. For $i\in\{0,\ldots,q\}$ let $H_i=(P_n\cprod P_2)[X_i]$ be the subgraph induced by $X_i$ and note that $H_i$ is an isometric subgraph of $P_n\cprod P_2$. If $0\leq i\leq q-1$ then $H_i\cong P_d\cprod P_2$ and $H_q\cong P_r$ (where $P_0$ is the empty graph). By (\ref{item_short_grids}) (which we already proved above), it follows from $k-1\leq d<2k-3$ that $\gp^k_d(P_d\cprod P_2)= 2(k-2)$. Also by (\ref{item_short_grids}), since $0\leq r<d<2k-3$, we see that $\gp^k_d(P_r\cprod P_2)=2\min(r,k-2)$. Therefore, by Lemma \ref{lemma_isometric_upper_bound} we obtain
\begin{align*}
    \gp^k_d(P_n\cprod P_2)&\leq \sum_{i=0}^q\gp^k_d(H_i)\\
    &=\gp^k_d(P_d\cprod P_2)q+\gp^k_d(P_r\cprod P_2)\\
    &=2(k-2)\floor{\frac{n}{d}}+2\min(n\bmod d,k-2).
\end{align*}

Let us prove (3). Suppose $n\geq d+1$ and $d\geq 2k-3$. By Lemma \ref{lemma_Akd_and_Bkd}, both $A_{k,d}\cap V(P_n\cprod P_s)$ and $B_{k,d}\cap V(P_n\cprod P_2)$ are $k$-general $d$-position sets and thus, given the cardinalities stated in Lemma \ref{lemma_Akd_and_Bkd}, we obtain
\begin{align*}
    \gp^k_d(P_n\cprod P_2)\geq\max(2(k-2)&\floor{\frac{n}{d}} + 2\min(n\bmod d,k-2), \\
         & (2k-3)\floor{\frac{n}{d}}+\min(n\bmod d,2k-3)).
\end{align*}
This implies that, in each sub-case of (3), $\gp^k_d(P_n\cprod P_r)$ is greater than or equal to the right-hand-side of the stated equation.

Let us prove the reverse inequality in each case of (3). For notational simplicity, let $q=\floor{\frac{n}{d}}$ and $r=n\bmod d$. 

Let us now finish the proof of (3ai) and (3bi) with a single argument. Suppose $r<2k-3$. We must prove that
\[\gp^k_d(P_n\cprod P_2)\leq (2k-3)q+r.\]
For the sake of contradiction, suppose $A$ is a $k$-general $d$-position set in $P_n\cprod P_2$ such that $|A|>(2k-3)q+r$. Without loss of generality we can assume that $|A|=(2k-3)q+r+1$. By Lemma 
\ref{lemma_isometric_down_convex_down_and_up} and (1), it follows that $|A\cap X_i|\leq 2k-3$ for $i\in\{0,\ldots,q-1\}$ and thus $|A\cap X_q|=r+1$. Lemma \ref{lemma_isometric_down_convex_down_and_up} implies that $A\cap ([dq-d+1,dq+r]\times\{1,2\})$ is a $k$-general $d$-position set in the induced subgraph
\[(P_n\cprod P_2)[[dq-d+1,dq+r]\times\{1,2\}].\]
But this contradicts Lemma \ref{lemma_not_kgdp}.

Let us finish the proof of (3aii). Suppose $r\leq k-2$ and $r>q$. For the sake of contradiction, suppose $A$ is a $k$-general $d$-position set in $P_n\cprod P_2$ with $|A|>(2k-4)q+2r$. Without loss of generality, suppose that $|A|=(2k-4)q+2r+1$. It follows from (1) and the fact that for $i\in\{0,\ldots,q-1\}$, $A\cap X_i$ is a $k$-general $d$-position set in $H_i$, that
\[|A\cap ([dq+1,dq+r]\times\{1,2\})|\geq |A|-(2k-3)q=2r+1-q\geq r+2.\]
The rest of the argument is the same as that of (3ai) and (3bi).

Let us finish the proof of (3bii). Suppose $2k-4-r>q$. For the sake of contradiction, suppose $A$ is a $k$-general $d$-position set in $P_n\cprod P_2$ with $|A|=(2k-4)(q+1)+1=(2k-4)q+2k-3$. It follows from (1) and the fact that for $i\in\{0,\ldots,q-1\}$, $A\cap X_i$ is a $k$-general $d$-position set in $H_i$, that
$|A\cap X_i|\leq 2k-3$ for $i\in\{0,\ldots,q-1\}$
and thus
\[|A\cap ([dq+1,dq+r]\times\{1,2\})|\geq |A|-(2k-3)q=-q+2k-3\geq r+1.\]
The rest of the argument is the same as that of (3ai), (3bi) and (3aii).

Finally, what remains to prove of (3c) is a straightforward consequence of Lemma \ref{lemma_isometric_upper_bound}.
\end{proof}

\begin{remark}\label{remark_Moore_identity}
Let us make note of a curious identity that hints at some connection between the two-dimensional $k$-general $d$-position subsets of $P_n\cprod P_2$ on the one hand and the one-dimensional $2k-2$-general $d-1$-position, as well as $k-1$-general $d-1$-position subsets of $P_n$ on the other. By inspecting the formulas above, it is apparent that when $n\geq d+1$ and $d\geq 2k-3$ we have 
\[\gp^k_d(P_n\cprod P_2)=\max(\gp^{2k-2}_{d-1}(P_n),2\gp^{k-1}_{d-1}(P_n)).\]
Is there a way to use the formula for $\gp^k_d(P_n)$ (Theorem \ref{theorem_paths}) to prove the formulas for $\gp^k_d(P_n\cprod P_2)$ in Theorem \ref{theorem_thin_grids_Akd_Bkd} and Theorem \ref{theorem_gpkd_grids_part1}? If so, this might provide a strategy for proving formulas for $\gp^k_d(P_n\cprod P_3)$.
\end{remark}

\section{Open problems}\label{section_questions}

Many questions regarding $k$-general $d$-position subsets of graphs remain open. Let us provide some background information in order to motivate our first list of open problems. In Section \ref{section_thin_grids}, we found a formula for $\gp^k_d(P_n\cprod P_2)$. The general position number of toroidal graphs $\gp(C_n\cprod C_m)$ has been studied in \cite{MR4265041} and \cite{MR4585177}. Klavzar et. al. \cite{MR4265041} provided information on the general position number $\gp(P_n\cprod C_m)$ of cylinders. The general position number of the $n$-dimensional infinite grid $\P^{\Box,n}_\infty=P_\infty\cprod\cdots\cprod P_\infty$ has been studied in \cite{MR4154901}.

\begin{problems}
Find a formulas for the $k$-general $d$-position numbers of the following graphs in terms of the relevant parameters.
\begin{enumerate}
\item $P_n\cprod P_m$ for $m>2$ (2-dimensional finite grids, see Remark \ref{remark_Moore_identity})
\item $P_{n_1}\cprod \cdots\cprod P_{n_\ell}$ ($\ell$-dimensional finite grids) 
\item $C_n\cprod C_m$ (toroidal graphs)
\item $C_{n_1}\cprod \cdots\cprod C_{n_\ell}$ ($\ell$-dimensional toroidal graphs)
\item $P_n\cprod C_m$ (cylinders)
\item $P^{\Box,n}_\infty$ ($n$-dimensional infinite grid)
\end{enumerate}
\end{problems}

Considering Corollary \ref{corollary_infinite_grid} and the results of \cite{MR4154901}, we make the following conjecture.

\begin{conjecture}
$\gp^k(P^{\Box,n}_\infty)=(k-1)^{2^{n-1}}$.
\end{conjecture}

Klav\v{z}ar, Rall and Yero \cite{MR4341189} proved the following structural characterization of general $d$-position sets. Recall that subgraphs $H_1$ and $H_2$ of a graph $G$ are \emph{parallel}, denoted by $H_1\parallel H_2$, if for every pair of vertices $u_1\in V(H_1)$ and $u_2\in V(H_2)$ we have $d_G(u_1,u_2)=d_G(H_1,H_2)$.

\begin{theorem}[{\cite{MR4341189}}]\label{theorem_characterization}
Let $G$ be a connected graph and let $d\geq 2$ be an integer. Then $S\subseteq V(G)$ is a general $d$-position set if and only if the following conditions hold:
\begin{enumerate}
\item The induced subgraph $G[S]$ is a disjoint union of complete graphs $Q_1,\ldots,Q_\ell$.
\item If $Q_i\nparallel Q_j$ where $i\neq j$, then $d_G(Q_i,Q_j)\geq d$.
\item If $d_G(Q_i,Q_j)+d_G(Q_j,Q_k)=d_G(Q_i,Q_k)$ for $\{i,j,k\}\subseteq\{1,\ldots,\ell\}$ with $|\{i,j,k\}|=3$, then $d_G(Q_i,Q_k)>d$.
\end{enumerate}
\end{theorem}

\begin{question}
Can one characterize $k$-general $d$-position subsets of connected graphs in a way which is similar to the characterization of general $d$-position sets given in Theorem \ref{theorem_characterization}?
\end{question}

Klav\v{z}ar, Rall and Yero \cite[Section 2]{MR4341189} showed that one can construct graphs with various configurations of strict and non-strict inequalities in the chain $\gp_2(G)\geq \gp_3(G)\geq\cdots\geq_{\diam(G)}(G)$. For example, if $T_n$ is the tree obtained from a path $P_{n+1}$ on $n+1$ vertices by attaching two leaves to each of its internal vertices then
\[\gp_2(T_n)=\gp_3(T_n)=\cdots=\gp_n(T_n).\]
Furthermore, they construct a graph $H_t$ with diameter $t$ such that
\[\gp_2(H_t)>\gp_3(H_t)>\cdots>\gp_t(H_t).\]

\begin{question}
To what extent can one construct graphs for which various configurations of strict and non-strict inequalities among the values of $\gp^k_d(G)$ are realized? Recall that the lattice of inequalities in Figure \ref{figure_lattice} holds in general. For example, is there a graph for which all the inequalities in Figure \ref{figure_lattice} are strict?
\end{question}

Klav\v{z}ar, Rall and Yero showed that the general $d$-position problem is NP-complete \cite[Theorem 2.2]{MR4341189}. We end with a conjecture regarding the computational complexity of the following decision problem related to finding the $k$-general $d$-position number of graphs.

\begin{quote}
\textsc{$k$-general $d$-position problem}
\begin{align*}
\text{\emph{Input}: \ \ } & \text{A graph $G$, an integer $k\geq 2$, an integer $d\geq 1$}\\
    & \text{and a positive integer $r$.}\\
\text{\emph{Question}: \ \ } & \text{Is $\gp^k_d(G)$ greater than $r$?}
\end{align*}
\end{quote}

\begin{conjecture}
The \textsc{$k$-general $d$-position problem} is \textrm{NP}-complete.
\end{conjecture}


\end{document}